\documentclass[12pt,a4paper]{article}
\usepackage[english]{babel}
\usepackage{url}
\usepackage{amssymb,amsmath,verbatim,graphicx}

\parskip\medskipamount
\newtheorem{theorem}{Theorem}[section]

\newtheorem{prop}[theorem]{Proposition}
\newtheorem{lemma}[theorem]{Lemma}
\newtheorem{cor}[theorem]{Corollary}
\newtheorem{observe}[theorem]{Observation}
\newtheorem{main}[theorem]{Main Result}
\newtheorem{rem}[theorem]{Remark}

\newenvironment{proof}{\par\noindent\textbf{Proof}\hspace{1em}}{\qed}
\newenvironment{proof*}{\par\noindent\textbf{Proof}\hspace{1em}}{}
\def\<{\langle}
\def\>{\rangle}

\newcommand{\cF}{\mathcal{F}}

\newcommand{\PG}{\mathsf{PG}}

\renewcommand{\L}{\mathbb{L}}
\newcommand{\F}{\mathbb{F}}

\newcommand{\cA}{\mathcal{A}}

\newcommand{\Res}{\mathrm{Res}}

\newcommand{\N}{\mathbb{N}}

\def\qed{{\hfill\hphantom{.}\nobreak\hfill$\Box$}}

\begin{document}

\author{Anneleen De Schepper$^1$ \and Hendrik Van Maldeghem$^2$}
\title{Graphs, defined by Weyl distance or incidence, that determine a vector space}
\date{\footnotesize $^{1,2}$ Department of Mathematics,\\
Ghent University,\\
Krijgslaan 281-S22,
B-9000 Ghent,\\
BELGIUM\\
$^1$ \texttt{Anneleen.DeSchepper@UGent.be}\\ $^2$ \texttt{Hendrik.VanMaldeghem@UGent.be}\\ $^2$ Corresponding author, tel. +32 9 264 49 11
}
\maketitle

\begin{abstract}
We study to which extent the family of pairs of subspaces of a vector space related to each other via intersection properties determines the vector space. In another language, we study to which extent the family of vertices of the building of a projective space related to each other via several natural respective conditions involving the Weyl distance and incidence determines the building. These results can be seen as generalizations of and variations on the Fundamental Theorem of Projective Geometry.    
\end{abstract}

{\footnotesize
}

\section{Introduction}
Let $\PG(n,\L)$ be an $n$-dimensional projective space over the skew field $\L$, i.e., the geometry of all nontrivial subspaces of an $(n+1)$-dimensional vector space $V_{n+1}(\L)$ over $\L$. 
\emph{The Fundamental Theorem of Projective Geometry} (see e.g.~\cite{BC}) states that, in graph-theoretical terms, every automorphism of the incidence graph of the point-line geometry associated to $\PG(n,\L)$, $n\geq 2$, is induced by a semi-linear permutation of the underlying vector space, or a duality (if $n=2$). In other words, every permutation of the 1-spaces of $V_{n+1}(\L)$ inducing a permutation of the $2$-spaces of $V_{n+1}(\L)$ is induced by a vector space (anti)automorphism.  In fact, all information of $\PG(n,\L)$ is already contained in the graph with the lines of $\PG(n,\L)$, or equivalently, the $2$-spaces of $V_{n+1}(\L)$, as vertices, where two vertices are adjacent if the corresponding subspaces intersect nontrivially. This is the collinearity graph of the line Grassmannian of the projective space. More generally, every automorphism of the collinearity graph of \emph{any} Grassmannian of $\PG(n,\L)$ is induced by a semi-linear permutation of the underlying vector space $V_{n+1}(\L)$, or a duality thereof, by a fundamental result of Chow \cite{Cho:49}. 

In the present paper, we generalize this by considering the subspaces of other arbitrary dimensions. Since we will only use subspaces of $V_{n+1}(\L)$, and not the vectors themselves, we prefer to work in the projective setting and hence use projective dimensions of the subspaces (1-spaces of $V_{n+1}(\L)$ are points or 0-spaces of $\PG(n,\L)$, $2$-spaces of $V_{n+1}(\L)$ are lines or $1$-spaces of $\PG(n,\L)$, $3$-spaces of $V_{n+1}(\L)$ are planes or $2$-spaces of $\PG(n,\L)$, \ldots, $k$-spaces of $V_{n+1}(\L)$ are $(k-1)$-spaces of $\PG(n,\L)$, $0\leq k\leq n+1$; note that the $0$-space of $V_{n+1}(\L)$ is a $-1$-space of $\PG(n,\L)$). So we consider the bipartite graph $\Gamma^n_{i,j;k}(\L)$ of $i$- and $j$-spaces of $\PG(n,\L)$, where an $i$-space is adjacent to a $j$-space if their intersection is a $k$-space, $-1\leq k\leq i\leq j\leq n$.  This is a metric generalization of the set-up in the previous paragraph, because ``intersecting in a $k$-space'' is a well-defined Weyl distance (a double coset in the Weyl group, which is the symmetric group on $n+1$ letters) between $i$-spaces and $j$-spaces, when $\PG(n,\L)$ be viewed as a building, and its subspaces as vertices. From the incidence geometric point of view, however, the general set-up is the graph $\Gamma^n_{i,j;\geq k}(\L)$ of $i$-spaces and $j$-spaces of $\PG(n,\L)$, where an $i$-space is adjacent to a $j$-space if they are both incident or equal to a common $k$-space. This means that their intersection has dimension \emph{at least} $k$, whence the notation. 

We show that, if $\Gamma^n_{i,j;k}(\L)$ or $\Gamma^n_{i,j;\geq k}(\L)$ are not trivial (meaning not complete bipartite, nonempty, not a matching and not the complement of a matching), then they completely determine the structure of $\PG(n,\L)$. In particular, every automorphism is induced by a semi-linear permutation of the underlying vector space, in some cases possibly a duality. Moreover, we will see that, apart from obvious isomorphisms, all these graphs are pairwise non-isomorphic, i.e., the values $n,i,j,k$ and the skew field $\L$ are essentially determined by the respective graphs (``essentially'' here means ``up to certain dualities'' which will be explained below). 

The metric point of view fits into the framework of the classical Theorem of Beckman and Quarles \cite{Bec-Qua:53} which states that the preservation of one single distance guarantees an isometry of the Euclidean plane, or more generally, the $n$-dimensional Euclidean space. Here we show that the preservation of one single Weyl distance between vertices on a building of type $\mathsf{A}_n$ guarantees an automorphism of the building. A similar result exists for  distances between chambers in any building \cite{Abr-Mal:04}. For vertices, only partial results exist, mainly treating the special case of largest distance (see below), or buildings in low rank (see \cite{Gov-Mal:02}). 

The main consequence concerns the graphs $\Gamma^n_{j;k}(\L)$ and $\Gamma^n_{j;\geq k}(\L)$ of $j$-spaces of $\PG(n,\L)$ where two vertices are adjacent if they intersect in a $k$-space and if their intersection contains a $k$-space, respectively. Indeed, since the (possibly extended) bipartite doubles of these graphs are $\Gamma^n_{j,j;k}(\L)$ and $\Gamma^n_{j,j;\geq k}(\L)$, respectively, every graph automorphism of $\Gamma^n_{j;k}(\L)$ or of $\Gamma^n_{j;\geq k}(\L)$ is induced by a semi-linear permutation or a duality of the underlying vector space, as soon as these graphs are not trivial (meaning not empty and not complete). For the graphs $\Gamma^n_{j;\geq k}(\L)$, this has been shown by Lim \cite{Lim:10} with a beautiful geometric argument. In fact, Lim considers adjacency preserving (in both directions) \emph{surjections} of the graphs. This slightly weaker hypothesis also suffices in our setting, but we find it more convenient to work with bijections and afterwards deduce this slightly more general result, see Remark~\ref{surjections}. We will not need to use Lim's results and thus provide an alternative approach to Lim's theorem. 

The proof in the case of  $\Gamma^n_{i,j;\geq k}(\L)$ uses the idea of a round-up triple introduced in \cite{Kas-Mal:13}, where opposition is handled (and ``opposition'' is just the maximal Weyl distance, using the ``longest word''). In fact the idea of a round-up triple is a more conceptual way to formulate Lim's proof, and it allows to treat more general situations. But basically, Lim's proof and ours are very alike, when applied to $\Gamma^n_{j;\geq k}(\L)$. Our results can also be seen as the completion of Lim's results in the most general case.

But the method of round-up triples does not work anymore for the graphs  $\Gamma^n_{i,j;k}(\L)$. This is rather surprising since a similar idea for opposite chambers, see \cite{Abr-Mal:00} carried over to single Weyl distance between chambers, see \cite{Abr-Mal:04}. For  $\Gamma^n_{i,j;k}(\L)$, we have to use ``round-up quadruples'', which considerably complicates things, and excludes $|\L|=2$. Here, Lim's approach does not work anymore.

This brings us to the special case of finite $\L$, where our results can be proved using the classification of maximal subgroups of the symmetric and alternating groups in \cite{Lie-Pre-Sax:87}. But that proof does not give much insight into the problem, of course. %We also include the \emph{thin case}, i.e., $|\L|=1$, which can also be proved using \cite{Lie-Pre-Sax:87}. 
In the finite case, we will denote the graphs $\Gamma^n_{i,j;k}(\L)$ and $\Gamma^n_{i,j;\geq k}(\L)$ by $\Gamma^n_{i,j;k}(|\L|)$ and $\Gamma^n_{i,j;\geq k}(|\L|)$, respectively, since a finite field is determined by its order. 

%As a corollary we obtain that, as soon as the graph $\Gamma^n_{j;i}(\L)$ of $j$-spaces, with adjacency coinciding with intersecting in an $i$-space, is nontrivial (nonempty and not complete), then every automorphism  is induced by a semi-linear permutation of the underlying vector space, possibly a duality. Another special case is the graph  $\Gamma^n_{j;\geq i}(\L)$ of $j$-spaces, with adjacency coinciding with the intersection containing an $i$-space. If such a graph is nontrivial then every graph automorphism is induced by a semi-linear permutation of the underlying vector space, and this has been shown by Lim \cite{Lim:10}.  Other corollaries and special cases will be mentioned below. 

%The graphs introduced above can also be defined purely in the language of vector spaces. For instance, $\Gamma^n_{i,j;k}(\L)$ is the bipartite graph with vertex set the subspaces of dimension $i+1$ and $j+1$ of a vector space $V_{n+1}(\L)$ of dimension $n+1$ over the skew field $\L$, where a $(j+1)$-space is adjacent to an $(i+1)$-space if their intersection is $(k+1)$-dimensional. If that graph is nontrivial in the above sense, then our results imply that every automorphism of that graph is induced by a semi-linear mapping of $V_{n+1}(\L)$, or in some cases possibly a duality, too.

Apart from the results of Lim and Chow mentioned above, another special case of our results concerns the case  $\Gamma^n_{i,j;-1}(\L)$ (or the bipartite complement $\Gamma^n_{i,j;\geq 0}(\L)$). This has been treated by Blunck and Havlicek \cite{Blu-Hav:05}. We are not aware of other special cases in the literature. However, for polar spaces, a lot of similar problems have been solved, but thus far not in such a full generality as we do for projective spaces in the present paper (see the references in \cite{Lim:10}). This is our main motivation: settle the problem for projective spaces and polar spaces in the most general way. In the present paper, we deal with the projective spaces, and the analogue for polar spaces is work in progress. 

This research is part of a larger programme to determine all situations for spherical buildings \cite{Tits} where the family of pairs of vertices at certain fixed Weyl distance, or connected by an incidence condition, uniquely defines the building in question. Note that this is not always true (so the projective spaces are, in this respect, a nice class of spherical buildings), see for instance \cite{Kas-Mal:13}, where the maximal distance between maximal singular subspaces in certain parabolic quadrics is a counter example, or \cite{Gov-Mal:02}, where again the maximal distance between points of a generalized hexagon is a counter example. In the present paper we will find other such examples, be it in the thin case (but one of the examples is strongly related to the smallest thick generalized quadrangle).

Indeed, in the last section of the present paper, we also prove the analogue of our main results for the \emph{thin} case, i.e., for vector spaces or projective spaces over the field of order 1, hence just sets. For finite sets, everything will follow, with some additional work, from a group-theoretic result of Liebeck, Praeger and Saxl \cite{Lie-Pre-Sax:87}, but we also consider infinite sets. As mentioned in the previous paragraph, it is interesting to note that for finite sets, there are counter examples, i.e., there are situations where not all graph automorphisms are induced by a permutation of the starting set. For more details, see Section~\ref{finitethin}

\section{Statements of the results}

We now provide the exact statements alluded to in the introduction. Let $V_{n+1}(\L)$ be an $(n+1)$-dimensional right vector space over the skew field $\L$, and let $\PG(n,\L)$ be the associated projective space, i.e., the points of $\PG(n,\L)$ are the 1-spaces of $V_{n+1}(\L)$, and a $k$-space of $\PG(n,\L)$, $0\leq k\leq n-1$, consists of the $1$-spaces contained in subspace of dimension $k+1$ of $V_n(\L)$. The \emph{empty subspace} of $\PG(n,\L)$ corresponds to the trivial subspace of $V_n(\L)$ and has project dimension $-1$. 

We define the graphs $\Gamma^n_{i,j;k}(\L)=\Gamma^n_{j,i;k}(\L)$ and $\Gamma^n_{i,j;\geq k}(\L)=\Gamma^n_{j,i;\geq k}(\L)$ as above. The \emph{bipartite complement} of a bipartite graph $\Gamma$ is the graph obtained from $\Gamma$ by interchanging edges and non-edges between the biparts, while keeping no edges within the biparts. 

\begin{main}\label{main1}
Let $\L$ and $\L'$ be two skew fields, and let $-1\leq k\leq i\leq j\leq n-1$, $-1\leq k'\leq i'\leq j'\leq n'-1$ be integers, with $-1\notin\{i,i'\}$ and $n\geq 2$. 
\begin{itemize}
\item[$(i)$] If $i=j=k\geq 0$, then $\Gamma^n_{i,j;k}(\L)$ is a matching.
\item[$(ii)$] If $i=j=n-1=k+1$ or $i=j=0=k+1$, then $\Gamma^n_{i,j;k}(\L)$ is the complement of a matching.
\item[$(iii)$]  If $n+k<i+j$, then $\Gamma^n_{i,j;k}(\L)$ is an empty graph (a graph with vertices but no edges).
\item[$(iv)$] If $\L\cong\L'$, $n=n'$, $i'=n-1-j$, $j'=n-1-i$ and $k'=n-1+k-i-j$, then $\Gamma^n_{i,j;k}(\L)\cong \Gamma^{n'}_{i',j';k'}(\L')$.
\item[$(v)$] If $i+j\leq n-1$, $(i,j)\neq(0,0)$, $k<j$  and $i'+j'\leq n'-1$, then $\Gamma^n_{i,j;k}(\L)\cong\Gamma^{n'}_{i',j';k'}(\L')$ if and only if $\L\cong \L'$ and $(i,j,k,n)=(i',j',k',n')$. In this case every graph isomorphism is induced by a semi-linear bijection from $V_n(\L)$ to $V_{n'}(\L')$, or possibly to $V_{n'}^*(\L')$ (the dual of $V_{n'}(\L')$) if $i+j=n-1$ and $\L'\cong (\L')^*$ (the latter is the opposite skew field). 
\end{itemize} 
\end{main}

Note that the restrictions in $(v)$ are justified by $(iv)$, so that we really cover all possible cases. The same thing holds for the next result, where the restrictions in $(iv)$ are justified by $(iii)$.

\begin{main}\label{main2}
Let $\L$ and $\L'$ be two skew fields, and let $-1\leq k\leq i\leq j\leq n-1$, $-1\leq k'\leq i'\leq j'\leq n'-1$ be integers, with $-1\notin\{i,i'\}$ and $n\geq 2$. 
\begin{itemize}
%\item[$(i)$] If $i=k$, then $\Gamma^n_{i,j;\geq k}(\L)\cong \Gamma^n_{i,j;k}(\L)$.
%\item[$(ii)$] If $k=0$, then $\Gamma^n_{i,j;\geq k}(\L)$ is the bipartite complement of $\Gamma^n_{i,j;-1}(\L)$.
\item[$(i)$]  If $n+k\leq i+j$ or $k=-1$, then $\Gamma^n_{i,j;\geq k}(\L)$ is a complete bipartite graph.
%\item[$(iii)$] If $\L=\L'$, $n=n'$, $i'=n-1-j$, $j'=n-1-i$ and $k'=n-1+k-i-j$, then $\Gamma^n_{i,j;\geq k}(\L)\cong \Gamma^{n'}_{i',j';\geq k'}(\L')$.
\item[$(ii)$] If $k=i+j+1-n$, then $\Gamma^n_{i,j;\geq k}(\L)$ is the bipartite  complement of $\Gamma^n_{i,j;k-1}(\L)$; if $i=k$, then $\Gamma^n_{i,j;\geq k}(\L)\cong\Gamma^n_{i,j;k}(\L)$.
\item[$(iii)$] If $\L\cong\L'$, $n=n'$, $i'=n-1-j$, $j'=n-1-i$ and $k'=n-1+k-i-j$, then $\Gamma^n_{i,j;\geq k}(\L)\cong \Gamma^{n'}_{i',j';\geq k'}(\L')$.
\item[$(iv)$] If $i+j\leq n-1$, $-1\neq k<j$ and $i'+j'\leq n'-1$, then $\Gamma^n_{i,j;\geq k}(\L)\cong\Gamma^{n'}_{i',j';\geq k'}(\L')$ if and only if $\L\cong \L'$ and $(i,j,k,n)=(i',j',k',n')$. In this case every graph isomorphism is induced by a semi-linear bijection from $V_n(\L)$ to $V_{n'}(\L')$, or possibly to $V_{n'}^*(\L')$ (the dual of $V_{n'}(\L')$) if $i+j=n-1$ and $\L'\cong (\L')^*$ (the latter is the opposite skew field). 
\end{itemize} 
\end{main}

We will also show that all graphs in $(v)$ of Main Result~\ref{main1} are distinct from those in $(iv)$ of Main Result~\ref{main2}, except for $k=\min\{i,j\}$ in both. 

A special case worth mentioning is the graph $\Gamma^n_{i,j;i}(\L)$, with $0\leq i<j\leq n-1$. This is the bipartite graph consisting of $i$-subspaces and $j$-subspaces, where adjacency is just defined by containment. It leads to the most straightforward generalization of the Fundamental Theorem of Projective Geometry, and, as we will see,  we will need to prove it separately in advance. % (but it follows from Lim's result \cite{Lim:10}).

Another special case occurs when $i=j$; in this case one can define the non-bipartite graphs $\Gamma^n_{j;k}(\L)$ and $\Gamma^n_{j;\geq k}(\L)$ (see above), see Corollaries~\ref{th1} and~\ref{th2}. 

\textbf{Notation.} The incidence graph of $\PG(n,\L)$ is an $n$-partite graph denoted by $\Gamma^n_{[0,n-1]}(\L)$. If we restrict this graph to the subspaces of dimensions $i,i+1,\ldots, j-1,j$, then we denote the resulting $(j-i+1)$-partite graph by $\Gamma^n_{[i,j]}(\L)$. In such a graph, the $k$-neighbor of a vertex $v$, with $i\leq k\leq j$ and with $\dim(v)\neq k$, are the subspaces of dimension $k$ incident with $v$. The $k$-neighborhood of $v$ is the set of $k$-neighbors of $v$. 

For $0\leq j\leq n-1$, the $j$-Grassmann graph is the graph with vertices the $j$-spaces of $\PG(n,\L)$, where two $j$-spaces are adjacent if they intersect in a $(j-1)$-space. This is the collinearity graph of the so-called $j$-Grassmannian geometry, which is defined as follows. The points are the $j$-spaces and the lines are the sets of $j$-spaces containing a fixed $(j-1)$-space $J_-$ and being contained in a fixed $(j+1)$-space $J_+$, with $J_-\subseteq J_+$.  If $j\notin\{0,n-1\}$, then the $j$-Grassmann graph uniquely determines the $j$-Grassmannian geometry. In any case, the $j$-Grassmannian geometry completely determines $\PG(n,\L)$ in the sense that the automorphism groups of both structures coincide (possibly up to the dualities). 

For a set $S$ of subspaces (possibly just points), we define $\<S\>$ to be the subspace generated by all members of $S$. If $S$ consists of two distinct points $p_1,p_2$, then we also denote the unique line passing through these points by $p_1p_2$.  Finally, for a $k$-subspace $K$, we denote by $\Res(K)$ the projective space of dimension $n-k-1$ obtained from the underlying vector space by factoring out $K$, and we call it the \emph{residue of $K$}. Hence the $i$-spaces of $\Res(K)$, $-1\leq i\leq n-k-1$, are the quotients $W/K$, where $W$ is an $(i+k+1)$-space of $\PG(n,\L)$ containing $K$. 

\section{Proofs}
\subsection{Generalities and the case $k=0$}
The assertions $(i)$ to $(iv)$ of Main Result~\ref{main1} and $(i)$ to $(iii)$ of Main Result~\ref{main2} are easy to verify. Hence we concentrate on showing $(v)$ of Main Result~\ref{main1} and $(iv)$ of Main Result~\ref{main2}. 

So let there be given a graph $\Gamma\cong\Gamma^n_{i,j;k}(\L)$ (with $n,i,j,k,\L$ as in $(v)$ of Main Result~\ref{main1}) or $\Gamma\cong\Gamma^n_{i,j;\geq k}(\L)$ (with $n,i,j,k,\L$ as in $(iv)$ of Main Result~\ref{main2}), except that we do not assume that $i\leq j$. We provide an algorithmic proof, determining the parameters as we go along. An exception is the family of graphs $\Gamma\cong\Gamma^n_{i,j;k}(2)$, $k<\min\{i,j\}$, which we must handle separately. Hence we will assume that $\Gamma$ is not isomorphic to such a graph. 

In the course of the proof, we will have to pick at certain moments a subspace of certain dimension satisfying different incidence conditions. We first prove a lemma that will imply that we can do so in the most restrictive case (which will take care of all other cases, too, that we will encounter).

\begin{lemma}\label{choose}
Let $a\geq 2$ and let $0\leq b< a$ be natural numbers. Let $B$ be $b$-space in $\PG(a,\L)$, let $B_1,B_2$ be two subspaces of $\PG(a,\L)$ of dimension at most $b-1$, and let $B_3$ be a subspace of $\PG(a,\L)$ of dimension at most $b-2$ (if $b=0$, then $B_3$ is the empty space). Then there exists an $(a-b-1)$-space $C$ disjoint from $B\cup B_1\cup B_2\cup B_3$. 
\end{lemma}

\begin{proof}
We assume $a\geq 3$, as the case $a=2$ is easy (there is always a point not incident with a given line and distinct from two other given points, and there is always a line not incident with a given point). By possibly extending the spaces $B_1,B_2,B_3$, we may assume that they have precisely dimension $b-1,b-1,b-2$, respectively. 

We claim that $B\cup B_1\cup B_2\cup B_3$ is nonempty. Indeed, if $|\L|$ is infinite, then this follows from the fact that $\PG(a,\L)$ is not the union of a finite number of (proper) subspaces. Now let $|\L|=q$ be finite. We may extend the subspaces $B,B_1,B_2,B_3$ in such a way that they have maximal dimension, i.e., we may assume that $b=a-1$. Then, taking into account that $B\cap B_1$ is at least $(a-2)$-dimensional, and similarly for $B\cap B_2$ and $B\cap B_3$, we count at most
$$\underbrace{\frac{q^a-1}{q-1}+\frac{q^{a-1}-1}{q-1}-\frac{q^{a-2}-1}{q-1}}_{B\cup B_1}+\underbrace{\frac{q^{a-1}-1}{q-1}-\frac{q^{a-2}-1}{q-1}}_{B_2\setminus B}+\underbrace{\frac{q^{a-2}-1}{q-1}-\frac{q^{a-3}-1}{q-1}}_{B_3\setminus B}$$ points in $B\cup B_1\cup B_2\cup B_3$. The claim follows if we show that this number is strictly less than $\frac{q^{a+1}-1}{q-1}$, which follows immediately from the obvious inequality $$q^{a+1}>q^a+2q^{a-1}-q^{a-2}-a^{q-3},$$ for all $a\geq 3$ and all $q\geq 2$. 

Now we continue by induction on $a\geq b+1$. The case $a=b+1$ follows from the claim above (then $C$ is just a point outside $B\cup B_1\cup B_2\cup B_3$).   Let $a>b+1$. Let $x$ be a point outside $B\cup B_1\cup B_2\cup B_3$. Then by the induction hypothesis we obtain an $(a-b-2)$-space $C'$ in $\Res(x)$ disjoint from $\<B,x\>/x,\<B_1,x\>/x\<B_2,x\>/x,\<B_3,x\>/x$. The corresponding  subspace $C$ in $\PG(a,\L)$ intersects each of $B,B_1,B_2,B_3$ exactly in the point $x$. Hence $C$ is disjoint from $B\cup B_1\cup B_2\cup B_3$ and the lemma is proved. 
\end{proof}

%\begin{rem}\label{choosebis}\em
%Using the same argument as in the proof of Lemma~\ref{choose}, one easily shows that, with the oration of Lemma~\ref{choose}, for any pair of $b$-subspaces $B,B'$, there exists an $(a-b-1)$-subspace $C'$ disjoint from $B\cup B'$. 
%\end{rem}

Now we start by isolating the case $k=\min\{i,j\}$, $(i,j)\notin\{(0,n-1),(n-1,0)\}$, $|j-i|>1$. For a vertex $v$ of $\Gamma$, we denote by $\Gamma(v)$ the set of neighbors of $v$. Also, $V(\Gamma)$ is the set of vertices of $\Gamma$, and $V^v(\Gamma)$ is the set of vertices in the bipart of $v$. 

\begin{prop}\label{propmin}
For the graph $\Gamma$ the parameter $k$ equals $\min\{i,j\}$ with $(i,j)\notin\{(0,n-1),(n-1,0)\}$ and $|j-i|>1$, if and only if $\Gamma$ satisfies the  following property \emph{(min)}
\begin{itemize}
\item[\emph{(min)}] For some vertex $v$, the family $\cF(v)=\{\Gamma(v)\cap\Gamma(w):w \in V^v(\Gamma)\}$ forms a poset under inclusion with the property that, if two elements have a greatest common lower bound, then it is obtained by intersecting the two elements.   Also, every two maximal elements have a greatest common lower bound and every element is contained in a maximal element.
\end{itemize}
\end{prop}

\begin{proof} 
If $k=\min\{i,j\}$, say $i=k$, and $v$ is a $j$-space, then $\cF(v)$ is the poset of all subspaces of dimension at least $\max\{i,2j-n\}$, viewed as sets of the $i$-spaces they contain, of a projective space of dimension $j$. It is clear that this poset satisfies (min) as soon as $i< j-1$ and $2j-n< j-1$. The first condition is equivalent with $|j-i|>1$ and the second with $j<n-1$, so with $(i,j)\neq(0,n-1)$ in view of $i+j\leq n-1$.
% (if $v$ is a $j$-space), or 
%of dimension at most $\min\{i,j-i-1\}$, viewed as the set of the $(j-i-1)$-spaces they are contained in, of a projective space of dimension $n-i-1$. Dualizing, this is the poset of all subspaces of dimension at least $\max\{n-2i-2,n-j-1\}$, viewed as sets of $(n-j-1)$-spaces they contain, of a projective space of dimension $n-i-1$. In these cases, clearly Condition~(min) is satisfied. 

Now suppose $0\leq k<\min\{i,j\}$ (we cease to assume $i\leq j$ as we did in the previous paragraph). Let $J_1$ and $J_2$ be two $j$-spaces intersecting in a $(j-1)$-space $J'$. Let $J_3$ be such that $\Gamma(J_1)\cap\Gamma(J_3)$ is maximal in $\cF(J_1)$ and contains $\Gamma(J_1)\cap\Gamma(J_2)$. We claim first that $J_1\cap J_2\subseteq J_3$. Indeed, suppose not. Then we can select a $k$-subspace $K$ contained in $J_1\cap J_2$ but not in $J_3$. Applying Lemma~\ref{choose} in $\Res(K)$, we see that there exists an $i$-space $I$ intersecting both $J_1$ and $J_2$ in $K$, and intersecting $J_3$ in $J_1\cap J_2\cap K$, which has dimension strictly less than $k$. Hence $I\in\Gamma(J_1)\cap\Gamma(J_2)\setminus\Gamma(J_3)$, a contradiction.  Now there are two possibilities. 

Assume that $J_3\notin\<J_1,J_2\>$. Then there is a $(k+1)$-space $K$ in $\<J_1,J_2\>$ intersecting $J_1\cap J_2$ in a $(k-1)$-space. We can now pick an $i$-space $I$ through $K$ not intersecting $J_3\setminus K$. Hence $I$ is adjacent to both $J_1,J_2$, but not to $J_3$, contradicting the maximality of $\Gamma(J_1)\cap\Gamma(J_3)$ in $\cF(J_1)$. 

Hence we may assume $J_3\in\<J_1,J_2\>$. If $\Gamma(J_1)\cap\Gamma(J_2)=\Gamma(J_1)\cap\Gamma(J_3)$, then $\Gamma(J_1)\cap\Gamma(J_2)$ was already maximal in the first place. If $\Gamma(J_1)\cap\Gamma(J_2)\neq\Gamma(J_1)\cap\Gamma(J_3)$, then, since there exists a collineation fixing $J_1$ and interchanging $J_2$ and $J_3$, there is a vertex in $\Gamma(J_1)\cap\Gamma(J_2)$ that does not belong to $\Gamma(J_1)\cap\Gamma(J_3)$, again a contradiction. 

We conclude that $\Gamma(J_1)\cap\Gamma(J_2)$ is maximal itself. Now consider any $j$-space $J_4$ not contained in $\<J_1,J_2\>$ and such that $J_1\cap J_2\subseteq J_4$. As above, we know that $\Gamma(J_1)\cap\Gamma(J_2)\neq\Gamma(J_1)\cap\Gamma(J_4)$. Hence, by (min), there exists a $j$-space $J_5$ with $\Gamma(J_1)\cap\Gamma(J_5)=\Gamma(J_1)\cap\Gamma(J_2)\cap\Gamma(J_4)$. But, also as above, if $J_5$ does not contain $J_1\cap J_2$, then we can select a $k$-space $K$ in $J_1\cap J_2$ not contained in $J_5$, and an $i$-space $I$ through $K$ such that $I\cap J_5=I\cap K$, $I\cap J_i=K$, for $i\in\{1,2,4\}$. Then $I$ belongs to $\Gamma(J_1)\cap\Gamma(J_2)\cap\Gamma(J_4)\setminus\Gamma(J_5)$, a contradiction. Hence $J_1\cap J_2\subseteq J_5$. But then $\Gamma(J_1)\cap\Gamma(J_5)$ is maximal, the final contradiction. Hence (min) is not satisfied.  

Now let $k=-1$. Let $J_1,J_2,J_3$ be three different $j$-spaces. It is easy to see that through any point of $J_3\setminus (J_1\cup J_2)$ one can find an $i$-space disjoint from both $J_1$ and $J_2$. Hence every element of $\cF(J_1)$ is maximal and so (min) cannot be satisfied.   
\end{proof}

\begin{prop}\label{propk=i}
The parameters $i,j,n,\L$ are uniquely determined by the graph $\Gamma^n_{i,j;i}(\L)$, for $i\leq j\leq n-i-1$. Moreover, every graph automorphism is induced by a semi-linear permutation of the underlying vector space.
\end{prop}

\begin{proof}
Fix any vertex $v$. Then we consider the poset $P_v=\{\Gamma(v)\cap\Gamma(w_1)\cap\ldots\Gamma(w_t):w_1,\ldots,w_t \in V^v(\Gamma), t\in\N\}$. The length of a maximal chain in $P_v$ is precisely $j-i$.   We define a new graph $\Gamma'$ as follows. The vertices are the intersections of a finite number of neighborhoods of vertices of $V^v(\Gamma)$. Adjacency is containment made symmetric. It is clear that $\Gamma'$ is isomorphic to $\Gamma^n_{[i,j]}(\L)$. Now we can extend this graph ``at both ends'' as follows. We define a new graph $\Gamma^i$ where the vertices are the $i$-spaces, adjacent when they are adjacent in $\Gamma'$ to a common vertex representing an $(i+1)$-space. The graph $\Gamma^i$ is the $i$-Grassmann graph and hence it has two kinds of maximal cliques: all $i$-spaces contained in an $(i+1)$-space (these maximal cliques are visible in $\Gamma'$), and all $i$-space contained in an $(i-1)$-space. We add the latter maximal cliques to the graph $\Gamma'$ with natural adjacency between the new vertices and the $i$-spaces, and a new vertex $I'$ is adjacent to a vertex $L$ representing an $\ell$-space, $i+1\leq\ell\leq j$, if  $I'$ and $L$ are adjacent to a common $i$-space. We do the same ``at the other end'' (with $j$-spaces) and obtain the graph $\Gamma^n_{i-1,j+1}(\L)$, we keep doing this until every pair of the new vertices is adjacent to a common old vertex introduces at the previous step ``next to it''. Then we have $\Gamma^n_{[0,j+i]}(\L)$. This uniquely determines $i$ (and hence $j$, since we already knew $j-i$). But now we can extend this graph ``at the right'' $n-i-j-1$ times to obtain $\Gamma^n_{[0,n-1]}(\L)$, uncovering $n$. The Fundamental Theorem of Projective Geometry now applies, $\L$ follows and so does the proposition.    
\end{proof}

So we are left with the graphs $\Gamma^n_{i,j;i}(\F)$, $i+1= j\leq n-i-1$ or $(i,j)=(0,n-1)$, which we call of \emph{type} I; graphs $\Gamma^n_{i,j;\geq k}(\F)$ with $0\leq k<\min\{i,j\}$ and $i+j\leq n-1$, which we call of \emph{type} II; graphs $\Gamma^n_{i,j; k}(\F)$ with $-1\leq k<\min\{i,j\}$ and $i+j\leq n-1$, which we call of \emph{type} III.

Now we characterize the graphs $\Gamma^n_{i,j;i}(\F)$ with $i+1= j\leq n-i-1$, among these.

\begin{lemma}
Let $\Gamma$ be a graph of type \emph{I,II} or \emph{III}. Then $\Gamma\cong\Gamma^n_{i,j;i}(\F)$ with $i+1= j\leq n-i-1$ if and only if every vertex is the intersection $\Gamma(v)\cap\Gamma(w)$ of the neighborhoods of two vertices $v,w$. 
\end{lemma}

\begin{proof}
Clearly the property holds if $\Gamma\cong\Gamma^n_{i,j;i}(\F)$, with $i+1= j\leq n-i-1$. Now suppose $\Gamma$ satisfies the stated property. Then $\Gamma$ is not of type II or III since through any given $k$-space $K$ contained in two $j$-spaces $J_1,J_2$, one can select at least two $i$-spaces intersecting two given $j$-spaces only in $K$ (indeed, projecting from $K$, this amounts to choose two $(i-k-1)$-spaces disjoint from two given $(j-k-1)$-spaces in $(n-k-1)$-dimensional space, which can clearly be done). 
\end{proof}

By Proposition~\ref{propk=i}, every automorphism  of $\Gamma^n_{i,i+1;i}(\F)$ is induced by a semi-linear mapping of the underlying vector space, or a duality (which happens if $2i+2=n$).

From now on, we do not call a graph $\Gamma^n_{i,i+1;i}(\F)$ of type I anymore, since we already characterized it.

%We start by proving Main Result~\ref{main2}, as it is somewhat simpler than the proof of Main Result~\ref{main1}. 

%\subsection{Proof of Main Result~\ref{main2}}
%The assertions $(i)$ to $(iv)$ are easy to verify. Hence we concentrate on showing $(v)$. 
% Consider $\Gamma=\Gamma^{n}_{i,j;\geq k}(\L)$, with $i+j\leq n-1$, $0\leq k\leq i\leq j$. 
 
% We subdivide the proof in two cases.
 
% \textbf{First Case: $2k\geq i$.} 
 
% In this case the diameter of the graph is larger than 2 and we can define a new graph $\Gamma'$ on the $j$-subspaces by defining $\Gamma'$-adjacency as distance 0 or 2 in $\Gamma$. This is clearly nontrivial and isomorphic to $\Gamma^n_{j;\geq 2k-i}(\L)$ and the result follows from \cite{Lim:10}.
 
%\textbf{Second Case: $2k<i$.}
 
%We present a kind of ``algorithmic'' proof, i.e., we assume we are given the graph $\Gamma$ without knowing the parameters $i,j,k,n,\L$, and we try to recover them in a unique way. This will prove the assertion $(v)$. The method is to choose one of the bipartition classes of $\Gamma$ and consider certain triples of vertices satisfying a graph-theoretic condition. Our method will be independent of the choice of the bipartition class, but to fix the ideas, we explain it with the class of $j$-spaces. At the end we comment on the alternative choice. 

Let $\Gamma$ be a graph of type I, type II or type III. We introduce a property of triples and quadruples, respectively, of vertices in the same bipart. 

Suppose one of the biparts of $\Gamma$ are the $j$-spaces, and the other consists of the $i$-spaces. Let $J_1,J_2,J_3$ be three $j$-spaces. Then we say that $\{J_1,J_2,J_3\}$ is a \emph{$\Gamma$-round-up triple} if no $i$-space is $\Gamma$-adjacent to exactly two of $J_1,J_2,J_3$ and some $i$-space is $\Gamma$-adjacent to all of $J_1,J_2,J_3$. Also, $\{J_1,J_2,J_3\}$ is a \emph{regular round-up triple} if $J_1\cap J_2\cap J_3$ is a $(j-1)$-space and if $\<J_1,J_2,J_3\>$ is a $(j+1)$-space.

Now let $J_1,J_2,J_3,J_4$ be four $j$-spaces in $\PG(n,\L)$. Then we say that $\{J_1,J_2,J_3,J_4\}$ is a \emph{$\Gamma$-round-up quadruple} if every vertex that is $\Gamma$-adjacent to at least two among $J_1,J_2,J_3,J_4$ is adjacent to at least three of them, and some vertex is adjacent to at least three of them. Also, $\{J_1,J_2,J_3,J_4\}$ is called a \emph{regular round-up quadruple} if the four $j$-spaces all contain a fixed $(j-1)$-space and are themselves contained in a fixed $(j+1)$-space. %Since such quadruple can only exist if $|\L|>2$, we assume from now one $|\L|>2$. For finite fields, we refer to the next section. 
 
%It is easy to verify that regular round-up quadruples are $\Gamma$-round-up quadruples. Our main task now is to prove the converse.

We first investigate when $\Gamma$ contains $\Gamma$-round-up triples, then prove that, when they do, these are precisely the regular round-up triples. Using these regular round-up triples, we determine the parameters of the graph in a canonical way and show that $\Gamma$ determines $\PG(n,\L)$. 

Then we go on doing the same with $\Gamma$-round-up quadruples for the remaining graphs. 

\subsection{$\Gamma$-round-up triples}

%Our aim is to classify all $\Gamma$-round-up triples. In particular, we claim that these are precisely the triples of $j$-spaces. A latter triple will be called a \emph{regular round-up triple}. 

We start with a characterization of regular round-up triples and quadruples. % a lemma, which we state slightly more general than we need right now, but we will use the more general form later. Remember that $j\geq 1$.

\begin{lemma}\label{lem0}
Let $J_1,J_2,J_3,J_4$ be (not necessarily different) $j$-spaces such that the intersection of the distinct pairs is a fixed subspace $D$, say of dimension $d\geq-1$. Then $d=j-1$ and $\dim\<J_1,J_2,J_3,J_4\>=  j+1$ if and only if every line intersecting two different members of $\{J_1,J_2,J_3,J_4\}$, say $J_{\ell_1},J_{\ell_2}$, also intersects one of $J_{\ell_3},J_{\ell_4}$, with $\{\ell_1,\ell_2,\ell_3,\ell_4\}=\{1,2,3,4\}$.   
\end{lemma}

\begin{proof}
The lemma is trivial if $|\{J_1,J_2,J_3,J_4\}|\leq 2$. So we may assume that all of $J_1,J_2,J_3$ are distinct, and that either $J_4$ is distinct from all of $J_1,J_2,J_3$, or $J_3=J_4$. Also, if $d=j-1$ and $\dim\<J_1,J_2,J_3,J_4\>=  j+1$, then it is easy to see that the stated condition is satisfied.

Suppose now the stated condition is satisfied. If $j=0$, then the assertion is easy. Suppose now $j\geq 1$ and assume  for a contradiction that $d<j-1$. Then there is some line $L$ contained in $J_2\setminus D$. Consider any point $p_1$ in $J_1\setminus D$ and let $p_2,p_2',p_2''$ be three points on $L$. Two of the lines $p_1p_2,p_1p_2',p_1p_2''$ must then intersect either $J_3$ or $J_4$, say $J_3$, by the condition. Hence the plane $\<p_1,L\>$ intersects $J_3$ in a line $L'$. But then the point $L\cap L'$ belongs to $J_2\cap J_3$, hence to $D$, a contradiction. Hence  $d=j-1$.

The line joining a point of $J_1\setminus D$ with a point of $J_2\setminus D$ intersects $J_3\cup J_4$, clearly in a point not belonging to $D$. Hence at most one of $J_3,J_4$, say $J_4$, does not belong to $\<J_1,J_2\>$.  Similarly, one of $J_1,J_2$  belongs to $\<J_3,J_4\>$, say $J_1$. But then $J_4$ and $J_2$ belong to $\<J_1,J_3\>$, which implies $\dim\<J_1,J_2,J_3,J_4\>=\dim\<J_1,J_3\>=j+1$.
\end{proof}

If $\Gamma$ is of type I, then it is easy to verify that a $\Gamma$-round-up triple is either a set of three collinear points or a set of three hyperplanes containing the same $(n-2)$-space. 

Now suppose that $\Gamma$ is of type II, with $\Gamma\cong\Gamma^n_{i,j;\geq k}(\L)$, $k>-1$ (we do not assume $i\leq j$) or $\Gamma$ is of type III, with  $\Gamma\cong\Gamma^n_{i,j; k}(\L)$, $k\geq-1$. In these cases, we note the following equivalent more manageable condition for being a $\Gamma$-round-up triple. For type II, it follows immediately from the fact that, if $U,V$ are two subspaces of $\PG(n,\L)$, with $\dim V\leq i$, $\dim U=j$ and $\dim (U\cap V)<k$, then there is an $i$-space $W$ containing $V$ with $W\cap U=V\cap U$. Note that we will only state things for one choice of $(i,j)$; interchanging $i$ and $j$ usually results in another property, which we assume tacitly. 

\begin{observe}\label{obs}
A triple $\{J_1,J_2,J_3\}$ of $j$-spaces is a $\Gamma$-round-up triple, with $\Gamma$ of type \emph{II}, if and only if every subspace of dimension at most $i$ spanned by two $k$-spaces contained in two respective members of $\{J_1,J_2,J_3\}$ intersects the third member in a subspace of dimension at least $k$. 
\end{observe}

\begin{lemma}\label{lemextra}
A triple $\{J_1,J_2,J_3\}$ of $j$-spaces is a $\Gamma$-round-up triple, with $\Gamma$ of type \emph{III}, if and only if every subspace $L$ of dimension at most $i$ spanned by two $k$-spaces contained in two respective members of $\{J_1,J_2,J_3\}$ and intersecting these subspaces in the respective $k$-spaces,  intersects the third member in a subspace of dimension precisely $k$. 
\end{lemma}

\begin{proof}
The proof is a simplified version of the proof of Lemma~\ref{lem2} (see below). 

Assume for a contradiction that the assertion is false. By renumbering if necessary we may then assume that $K_1$ and $K_2$ are $k$-spaces contained in $J_1,J_2$, respectively, and that $\<K_1,K_2\>\cap J_3$ is not $k$-dimensional.
Our goal is to construct an $(i-1)$-space $I'$ containing $\<K_1,K_2\>$  and intersecting $J_3$ in a $(k-1)$-space. 

We consider an arbitrary $i$-space $I$ intersecting $J_1$ in $K_1$ and $J_2$ in $K_2$ (it is easy to see that such a subspace exists). Since $\{J_1,J_2,J_3\}$ is a $\Gamma$-round-up triple, we know that $I$ intersects $J_3$ in a $k$-space $K_3$. Since $K_3\not\subseteq\<K_1,K_2\>$ by our assumption, we can consider a hyperplane $I'$ of $I$ containing $\<K_1,K_2\>$ and not containing $K_3$. 

In $\Res(I')$, the spaces $\<J_1,I'\>/I'$, $\<J_2,I'\>/I'$ and $\<J_3,I'\>/I'$ have dimension $j-k-1$, $j-k-1$ and $j-k$, respectively. By Lemma~\ref{choose}, we can find a point in $\Res(I')$ which avoids these three subspaces. This point corresponds to an $i$-space $I^*$ intersecting $J_1$ and $J_2$ in $k$-spaces and intersecting $J_3$ in a $(k-1)$-space. This clearly contradicts the fact that $\{J_1,J_2,J_3\}$ is a $\Gamma$-round-up triple. 
\end{proof}

Suppose now $\{J_1,J_2,J_3\}$ is a $\Gamma$-round-up triple. In the next statements, the arguments are usually independent of the type of $\Gamma$; however, if they differ, then we write the argument for type II in the ordinary way, whereas the argument for type III will be written in square brackets. It will be clear which part of the argument for type II it replaces.

% \begin{lemma}\label{lem3bis}
%If $\{J_1,J_2,J_3\}$ is a $\Gamma$-round-up triple, then we can choose indices such that $\dim (J_1\cap J_2)\geq k$. 
 %\end{lemma}

%\begin{proof}
%Suppose $\dim(J_a\cap J_b)<k$, for all $a,b\in\{1,2,3\}$, $a\neq b$. Since $2k\leq i-1\leq j-1$, we can find a $k$-space $K_1$ in $J_1$ disjoint from $J_2\cup J_3$. Likewise, we can find two disjoint $k$-spaces $K_2,K_2'$ in $J_2$ disjoint from $J_1\cup J_3$. By Observation~\ref{obs}, the subspaces $\<K_1,K_2\>$ and $\<K_1,K_2'\>$ intersect $J_3$ in spaces $K_3$ and $K_3'$, respectively, of dimension at least $k$. Now $K_3$ and $K_3'$ are contained  in $\<K_1,K_2,K_2'\>$, which is $3k+2$-dimensional. Clearly, $K_3$ and $K_3'$ are disjoint, as otherwise their intersection contains a point of $K_1$, contradicting the choice of $K_1$ not containing a point of $J_3$. Hence a dimension argument inside $\<K_1,K_2,K_2'\>$ implies that $\<K_2,K_2'\>$ and $\<K_3,K_3'\>$ intersect in a space of dimension at least $2(2k+1)-(3k+2)=k$, contradicting the fact that $\dim (J_2\cap J_3)<k$. 
%\end{proof}

\begin{lemma}\label{lem1}
The $\Gamma$-round-up triple $\{J_1,J_2,J_3\}$ satisfies $J_1\cap J_2= J_3\cap J_1=J_2\cap J_3$.
\end{lemma}
  
\begin{proof}
Choose the subscripts such that $\dim(J_1\cap J_2)$ be maximal among $\{\dim(J_\ell\cap J_{\ell'}):\ell\in\{1,2,3\},\ell'\in\{1,2,3\}\setminus\{\ell\}\}$. For a contradiction, we may hence suppose that there exists a point $p\in (J_1\cap J_2)\setminus J_3$. There are two cases to distinguish.

\begin{enumerate}
\item Suppose $\dim(J_1\cap J_2)\geq k$. Consider a $k$-space $K$ in $J_1\cap J_2$ containing $p$. Observation~\ref{obs} implies that $K$ intersects $J_3$ in a subspace of dimension at least [exactly] $k$, which is absurd. 
\item Suppose $\dim(J_1\cap J_2)=\ell<k$.  Let $K_a$ be any $k$-space in $J_a$ containing $J_1\cap J_2$, $a=1,2$. By Observation~\ref{obs}, the subspace $\<K_1,K_2\>$ intersects $J_3$ in a subspace $K_3$ of dimension at least [exactly] $k$.  Since $\dim\<K_1,K_2\>=2k-\ell$ and $\<K_1,K_2\>$ contains $K_3$, a dimension argument implies $\dim K_1\cap K_3\geq \ell$, %, with equality only if $\dim K_3= k$. 
By maximality of $\ell$, we have equality. %, and so $\dim K_3=k$.
Moreover, this also implies $J_1\cap J_3=K_1\cap K_3$, and so $K_1\cap K_3\neq J_1\cap J_2$. Replacing $K_1$ by a $k$-space $K_1'$ in $J_1$ containing $J_1\cap J_2$ but not $K_1\cap K_3$, the same argument implies $J_1\cap J_3\subseteq K'_1$, a contradiction.
% over all $k$-spaces in $J_1$ containing $J_1\cap J_2$, we deduce that $J_1\cap J_3$ is contained in $J_1\cap J_2$, hence coincides with it.
\end{enumerate} 
The lemma is proved.
\end{proof}

\begin{lemma}\label{lem5}
Every $\Gamma$-round-up triple $\{J_1,J_2,J_3\}$ is a regular round-up triple.
\end{lemma}

\begin{proof*}
By Lemma~\ref{lem1}, we know $J_1\cap J_2=J_3\cap J_1=J_2\cap J_3$. Define $\ell=\dim J_1\cap J_2$. %By Lemma~\ref{lem3bis}, we know $\ell\geq k$. 
Again, we distinguish two cases.
\begin{enumerate}
\item  Suppose $\ell\geq k$. 
Let $p_a\in J_a\setminus J_3$ be arbitrary, $a=1,2$. For $a=1,2$, choose a $k$-space $K_a$ containing $p_a$ and intersecting $J_1\cap J_2$ in a $(k-1)$-space, with the latter independent of $a$. Observation~\ref{obs} [Lemma~\ref{lemextra}] and the fact that $k<i$ imply that the $(k+1)$-space $\<K_1,K_2\>$ intersects $J_3$ in at least [precisely] a $k$-space $K_3$ intersecting $J_1\cap J_2$ in $K_1\cap K_2$, hence in precisely a $k$-space. Consequently, the line $p_1p_2$ intersects $K_3$, and hence $J_3$, in a point. Lemma~\ref{lem0} shows the assertion.  
\item Suppose $\ell<k$. We claim this situation cannot occur. Indeed, let $K_a$ be a $k$-space through $J_1\cap J_2$ inside $J_a$, $a=1,2$. Observation~\ref{obs} [Lemma~\ref{choose}] and Lemma~\ref{lem1} together imply that $\<K_1,K_2\>\cap J_3=:K_3$ is $k$-dimensional. Now let $K_1'$ be a $k$-space in $J_1$ through $J_1\cap J_2$ such that $\<K_1,K'_1\>$ is $(k+1)$-dimensional (this is possible since $k<j$). Then $\<K_1',K_2\>\cap J_3=:K_3'$ is again $k$-dimensional. Since $\<K_2,K_3\>=\<K_1,K_2\>$ and $K_1'\not\subseteq\<K_1,K_2\>$, we conclude $K_3'\neq K_3$. Now $\dim\<K_1,K_1',K_2\>=2k-\ell+1$ and $\<K_3,K_3'\>\subseteq\<K_1,K_1',K_2\>$. A dimension argument (the Grassmann identity) implies now that $\<K_1,K_1'\>\cap\<K_3,K_3'\>$ has dimension at least $(k+1)+(k+1)-(2k-\ell+1)=\ell+1$. Since $\<K_1,K_1'\>\cap\<K_3,K_3'\>$ is contained in $J_1\cap J_3$, this is a contradiction.
%This in turn implies that, for any $(\ell+1)$-space $L_2$ in $J_2$ containing $J_1\cap J_2$, the subspace $L_3:=\<L_2,K_1\>\cap J_3$ is $(\ell+1)$-dimensional. Interchanging the roles of $J_2$ and $J_3$, we see that the mapping $L_2\mapsto L_3$ is a bijection from the set of $(\ell+1)$-spaces in $J_2$ through $J_2\cap J_3$ to the set of $(\ell+1)$-spaces in $J_3$ through $J_2\cap J_3$. Consequently $J_3\subseteq\<K_1,J_2\>$. A dimension argument implies now that $\ell=\dim (J_3\cap J_2)\geq 2j-(j+k-\ell)=j+\ell-k$, a contradiction. 
Our claim is proved.\qed
\end{enumerate}
\end{proof*}

\begin{lemma}\label{lemma5b}
A regular round-up triple is a $\Gamma$-round-up triple if and only if $\Gamma$ has type \emph{I} or \emph{II}.
\end{lemma}

\begin{proof}
The assertion is clear for graphs of type I and II. Now let $\Gamma\cong\Gamma^n_{i,j;k}(\L)$ be of type III. Let $\{J_1,J_2,J_3\}$ be a regular round-up triple of $j$-spaces. Since $k<j$, we can select a $k$-space $K$ in $J_1\cap J_2$. Since $k<i$, we can select an $i$-space $I$ intersecting $J_1\cap J_2$ in $K$, and intersecting $\<J_1,J_2\>$ in a $(k+1)$-space contained in $J_1$ (hence not in $J_1\cap J_2$). Then $I$ is adjacent with both $J_2$ and $J_3$, but not with $J_1$, a contradiction.  Note that this argument also works for $k=-1$.
\end{proof}

So we can distinguish graphs of type I and II from those of type III. 

We now make two further simplifications. 

First we split off the case $k=-1$ in type III graphs and forget about them, since their bipartite complements belong to type II.

\begin{lemma}\label{k=-1}
Let $\Gamma=\Gamma^n_{i,j;k}(\L)$ be a graph of type \emph{III}. Then $k=-1$ if and only if the bipartite complement $\Gamma^*$ of $\Gamma$ admits $\Gamma^*$-round-up triples.  
\end{lemma}

\begin{proof}
Clearly the bipartite complement of $\Gamma^n_{i,j;0}(\L)$ is $\Gamma^n_{i,j;\geq 0}(\L)$ and this admits $\Gamma^n_{i,j;\geq 0}(\L)$-round-up triples, namely, regular round-up triples. 

Now suppose $k\geq 0$. A $\Gamma^*$-round-up triple is, without loss, a triple $\{J_1,J_2,J_3\}$ of $j$-spaces such that no $i$-space intersects exactly one of $J_1,J_2,J_3$ in a $k$-space. We show that such a triple does not exist. Indeed, suppose for a contradiction that $\{J_1,J_2,J_3\}$ is such a triple. Then we select a point $p\in J_3\setminus(J_1\cup J_2)$ (this is possible since $J_3$ cannot be covered by the union of $J_1$ and $J_2$), and a $k$-space $K$ in $J_3$ containing $p$. It is easy to see that we can find an $i$-space $I$ intersecting $J_3$ in $K$ and intersecting $J_a$, $a=1,2$, in $K\cap J_a$. Indeed, it suffices to find in  $\Res(K)$ an $(i-k-1)$ space avoiding the three subspaces $\<J_1,K\>/K$, $\<J_2,K\>/K$ and $J_3/K$, which have dimension at most $j$, at most $j$, exactly $j-k-1$, respectively. 
\end{proof}

Hence, from now on, we assume that type III graphs have $k\neq -1$.

Now we distinguish type I from type II graphs. So let $\Gamma$ be a graph of type I or type II. Notice that it suffices to recognize when $i,j\in\{0,n-1\}$, since for type II graphs neither $i$ nor $j$ belongs to $\{0,n-1\}$ (since $0\leq k<\min\{i,j\}$). 

We define a new graph $\Gamma_1$ with vertex set, without loss, the set of $j$-spaces, with two $j$-spaces adjacent if they are contained in a $\Gamma$-round-up triple. This is the $j$-Grassmann graph of $\PG(n,\L)$, and so two $j$-spaces are contained in a unique maximal clique if and only if $j\in\{0,n-1\}$. Hence, if in $\Gamma_1$ every pair of distinct vertices is contained in a unique maximal clique, then $\Gamma$ has type I, and otherwise it has type II. If it has type I, then we can apply Proposition~\ref{propk=i}. Hence from now on, we only have to deal with graphs of type II and III.

We first handle the case of graphs of type II, thus finishing the proof of Main Result~\ref{main2}.

\subsection{Proof of Main Result~\ref{main2}}

The graph $\Gamma_1$ above determines $n,j,\L$, and likewise we can determine $i$. It does not directly determine $k$, although this can be proved indirectly. Nevertheless, we prefer to give a rather elegant algorithm to determine the parameters of $\Gamma$.

Remember that $\Gamma\cong\Gamma^n_{i,j;\geq k}(\L)$, with $0\leq k<\min\{i,j\}\leq n-\max\{i,j\}-1$. 

We define another graph $\Gamma_2$ with vertex set the set of $j$-spaces together with the set of maximal cliques of $\Gamma_1$. A $j$-space is adjacent to a maximal clique if it is contained in it, and two maximal cliques are adjacent if their intersection contains at least two $\Gamma_1$-adjacent $j$-spaces. This way, the tripartite graph isomorphic to $\Gamma^n_{[j-1,j+1]}(\L)$ of $(j-1)$-, $j$- and $(j+1)$-spaces arises with natural adjacency (namely, incidence). But we do not know yet from our data which classes precisely correspond to $(j-1)$-spaces, and which to $(j+1)$-spaces. %Now there are two possibilities.

We note that, since $k<i$, every $i$-space is $\Gamma$-adjacent to all members of maximal cliques of both kinds.
We define four different bipartite graphs $\Gamma^\epsilon_X$, with $\epsilon\in\{-1,+1\}$ and $X\in\{\exists,\forall\}$, with one class of vertices each time the set of $i$-spaces.  The other class of vertices of the two graphs $\Gamma^{-1}_X$, $X\in\{\exists,\forall\}$, is one class of $\Gamma_2$ distinct from the class corresponding to the $j$-spaces, and for the other two graphs we take the other class of $\Gamma_2$ distinct from the class of $j$-spaces. An $i$-space is  $\Gamma^\epsilon_\exists$-adjacent to a clique of $j$-spaces if it is adjacent to some member of the clique, $\epsilon\in\{-1,+1\}$, and an $i$-space is  $\Gamma^\epsilon_\forall$-adjacent to a clique of $j$-spaces if it is adjacent to all members of the clique, $\epsilon\in\{-1,+1\}$.  To fix the ideas, we can choose the notation in such a way that the superscript $-1$ corresponds to cliques of $j$-spaces containing a fixed $(j-1)$-space, and the superscript $+1$ to cliques of $j$-spaces contained in a fixed $(j+1)$-space (but this can for the moment not directly be derived from the graph $\Gamma$ and the information we have up to now). 

One easily checks the following isomorphisms:
 $$\begin{array}{ll}\Gamma^-_\exists\cong \Gamma^n_{i,j-1;\geq k-1}(\L), \hspace{1cm} & \Gamma^+_\exists\cong\Gamma^n_{i,j+1,\geq k}(\L),\\ [3mm] \Gamma^-_\forall\cong\Gamma^n_{i,j-1,\geq k}(\L),& \Gamma^+_\forall\cong\Gamma^n_{i,j+1,\geq k+1}(\L).\end{array}$$
Notice that we can repeat each of these four constructions for each of the four cases. But, starting from a minus graph, the appropriate plus graph brings us back to the original graph, and similarly for starting from a plus graph. Hence we can define with self-explaining notation the following graphs:

 $$\begin{array}{ll}\Gamma^{-m}_\exists\cong \Gamma^n_{i,j-m;\geq k-m}(\L), \hspace{1cm} &\Gamma^{+m}_\exists\cong\Gamma^n_{i,j+m,\geq k}(\L),\\ [3mm] \Gamma^{-m}_\forall\cong\Gamma^n_{i,j-m,\geq k}(\L),& \Gamma^{+m}_\forall\cong\Gamma^n_{i,j+m,\geq k+m}(\L),\end{array}$$

for all natural $m$ for which the right hand side makes sense. It is clear that the smallest $m$ for which  $\Gamma^{-m}_\exists$ becomes complete bipartite is $m=k+1$. Also, the smallest $m$ for which $\Gamma^{+m}_\exists$ becomes complete bipartite is $m=n+k-i-j$. Since $k+1\leq n+k-i-j$, we can deduce $k$ and $n-i-j$ from this information. 

It is also clear that the smallest $m$ for which $\Gamma^{-m}_\forall$ becomes an empty graph is $j-k+1$, and the smallest $m$ for which $\Gamma^{+m}_\forall$ becomes an empty graph is $i-k+1$. Form this, we deduce $j$ and $i$.

If we would have started with looking for $\Gamma$-round-up triples of $i$-spaces, then we would also have found first $k$, then $n-i-j-1$, and then $i-k+1$ and $j-k+1$, which also determines all parameters uniquely. 

Hence, in both cases, our method reveals $i,j,k$ and $n$ and Main Result~\ref{main2} follows from this and the fact that $\Gamma^n_{[j-1,j+1]}(\L)$ determines $\PG(n,\L)$ completely. 

%Notice that $j$ and $n$ also follow from $\Gamma^n_{[j-1,j+1]}(\L)$ by extending this graph ``at the left'' and ``at the right'' as we did before until we have the complete incidence graph of $\PG(n,\L)$. But the algorithm above is somehow more elegant. 

We mention the following special case of Main Result~\ref{main2}. Suppose $-1\leq k\leq j\leq n-1$, $j\neq -1$, and let $\Gamma^n_{j;\geq k}(\L)$ be the graph of $j$-spaces of $\PG(n,\L)$ where two $j$-spaces are adjacent if their intersection contains a $k$-space, and let $\Gamma^n_{j;k}(\L)$ be the graph of $j$-spaces of $\PG(n,\L)$ where two $j$-spaces are adjacent if they intersect in a $k$-space.

\begin{cor}\label{th1}
Let $\L$ and $\L'$ be two skew fields, and let $-1\leq k\leq j\leq n-1$, $-1\leq k'\leq j'\leq n'-1$ be integers, with $-1\notin\{j,j'\}$ and $n\geq 2$. 
\begin{itemize}
%\item[$(i)$] If $i=k$, then $\Gamma^n_{i,j;\geq k}(\L)\cong \Gamma^n_{i,j;k}(\L)$.
%\item[$(ii)$] If $k=0$, then $\Gamma^n_{i,j;\geq k}(\L)$ is the bipartite complement of $\Gamma^n_{i,j;-1}(\L)$.
\item[$(i)$]  If $n+k\leq 2j$ or $k=-1$, then $\Gamma^n_{j;\geq k}(\L)$ is a complete graph.
\item[$(ii)$] If $j=k$, then $\Gamma^n_{j;\geq k}(\L)$ is an empty graph.
\item[$(iii)$] If $\L\cong \L'$, $n=n'$, $j'=n-1-j$ and $k'=n-1+k-2j$, then $\Gamma^n_{j;\geq k}(\L)\cong \Gamma^{n'}_{j';\geq k'}(\L')$.
\item[$(iv)$] If $k=2j+1-n$, then $\Gamma^n_{j;\geq k}(\L)$ is the complement of $\Gamma^n_{j;k-1}(\L)$.
%\item[$(v)$] If $\L=\L'$, $n=n'$, $i'=n-1-j$, $j'=n-1-i$ and $k'=n-1+k-i-j$, then $\Gamma^n_{i,j;\geq k}(\L)\cong \Gamma^{n'}_{i',j';\geq k'}(\L')$.
\item[$(v)$] If $2j\leq n-1$, $-1< k<j$ and $2j'\leq n'-1$, then $\Gamma^n_{j;\geq k}(\L)\cong\Gamma^{n'}_{j';\geq k'}(\L')$ if and only if $\L\cong \L'$ and $(j,k,n)=(j',k',n')$. In this case every graph isomorphism is induced by a semi-linear bijection from $V_n(\L)$ to $V_{n'}(\L')$, or possibly to $V_{n'}^*(\L')$ (the dual of $V_{n'}(\L')$) if $2j=n-1$ and $\L'\cong (\L')^*$ (the latter is the opposite skew field). 
\end{itemize} 
\end{cor}

\begin{proof}
Recall that the (extended) bipartite double $2\Gamma$ ($\overline{2}\Gamma$) of a given graph $\Gamma$ is obtained by taking two copies of the vertex set $\Gamma$ (without the edges) and defining a vertex of one copy to be incident with a vertex of the other copy if the corresponding vertices of $\Gamma$ are (equal or) adjacent in $\Gamma$.

Then the corollary follows easily from Main Result~\ref{main2} and the observation that $\overline{2}\Gamma^n_{j;\geq k}(\L)\cong\Gamma^n_{j,j;\geq k}(\L)$ and $2\Gamma^n_{j;k}(\L)\cong\Gamma^n_{j,j; k}(\L)$ (the latter is needed for $(iv)$).
\end{proof}

We now go on with graphs of type III. 

\subsection{Proof of Main Result~\ref{main1}}
Now we prove Main Result~\ref{main1}. So $\Gamma\cong\Gamma^{n}_{i,j;k}(\L)$, with $i+j\leq n-1$ and $0\leq k<\min\{i,j\}$. Remember that $k\neq -1$ because of Lemma~\ref{k=-1}.  

% We again subdivide the proof in two cases.
 
% \textbf{First Case: $2k\geq i$.} 
 
% In this case the diameter of the graph is larger than 2 and we can define a new graph $\Gamma'$ on the $j$-subspaces by defining $\Gamma'$-adjacency as distance 0 or 2 in $\Gamma$. This is clearly nontrivial and isomorphic to $\Gamma^n_{j;\geq 2k-i}(\L)$ and the result follows from \cite{Lim:10} (see also Corollary~\ref{th1}$(v)$).
 
% \textbf{Second Case: $2k<i$.}
 
Our proof is again symmetric in $i$ and $j$. Hence throughout one can interchange $i$ and $j$.  
The method used in the previous subsection with round-up triples does not work because of Lemmas~\ref{lem5} and~\ref{lemma5b}. But we now use the idea of the $\Gamma$-round-up quadruples and show that every such quadruple is a regular round-up quadruple. This proof is considerably more involved, although the main structure is kept, and some results can be proved using a slight generalization. 

It is easy to show that every regular round-up quadruple is a $\Gamma$-round-up quadruple. Hence we show the converse.

We start with a lemma which can be seen as the counterpart to Observation~\ref{obs}. 

\begin{lemma}\label{lem2}
Let $\{J_1,J_2,J_3,J_4\}$ be a $\Gamma$-round-up quadruple. Then the following holds for arbitrary $\ell_1,\ell_2,\ell_3,\ell_4$ with $\{\ell_1,\ell_2,\ell_3,\ell_4\}=\{1,2,3,4\}$. Whenever $K_1$ and $K_2$ are $k$-spaces in $J_{\ell_1}$ and $J_{\ell_2}$, respectively, with $\<K_1,K_2\>\cap J_{\ell_a}=K_a$, for $a\in\{1,2\}$, and with $\dim\<K_1,K_2\>\leq i$, then the subspace $\<K_1,K_2\>$ intersects at least one of $J_{\ell_3},J_{\ell_4}$ in precisely a $k$-space.
\end{lemma}

\begin{proof}
Assume for a contradiction that the assertion is false. By renumbering if necessary we may then assume that $K_1$ and $K_2$ are $k$-spaces contained in $J_1,J_2$, respectively, and that none of $\<K_1,K_2\>\cap J_3$ and $\<K_1,K_2\>\cap J_4$ is $k$-dimensional.

Our main goal is to construct an $(i-1)$-space $I'$ containing $\<K_1,K_2\>$  and intersecting, say, $J_3$ in a $(k-1)$-space, while $\dim(I'\cap J_4)\neq k$. 

We consider an arbitrary $i$-space $I$ intersecting $J_1$ in $K_1$ and $J_2$ in $K_2$ (which exists by applying Lemma~\ref{choose} in $\Res(\<K_1,K_2\>)$). Since $\{J_1,J_2,J_3,J_4\}$ is a $\Gamma$-round-up quadruple, we know that $I$ intersects, say, $J_3$ in a $k$-space $K_3$. We will construct $I'$ as a hyperplane of $I$. Our construction depends on $\dim(I\cap J_4)$.

\textbf{First case: $\dim(I\cap J_4)\notin\{k,k+1\}$.} \\ This is an easy case. Indeed, first note that $K_3\not\subseteq\<K_1,K_2\>$ by our main assumption. We can thus consider a hyperplane $I'$ of $I$ containing $\<K_1,K_2\>$ and not containing $K_3$. 

\textbf{Second case: $I\cap J_4=K_4$ is $k$-dimensional.} \\ In this case we know that none of $K_3,K_4$ is contained in $\<K_1,K_2\>$, and hence $\<K_1,K_2\>$ is not a hyperplane of $I$, as otherwise $I'=\<K_1,K_2\>$ satisfies our needs. So we can first choose an $(i-2)$-space $I''$ through $\<K_1,K_2\>$ neither containing $K_3$, nor $K_4$. If $\<I'',K_3,K_4\>$ has dimension $i-1$, then  a hyperplane of $I$ through $I''$ distinct from $\<I'',K_3,K_4\>$ does the job. If $\<I'',K_3,K_4\>=I$, then we can consider a line $L$ intersecting $K_3$ and $K_4$ in respective distinct points and not intersecting $I''$. Now the space $\<I'',p\>$, with $p$ a point on $L$ not in $K_3\cup K_4$ is an $(i-1)$-space meeting our conditions.

\textbf{Third case: $\dim (I\cap J_4)=k+1$.} \\ Put $R_4=I\cap J_4$. If $K_3$ is not contained in $\<K_1,K_2,R_4\>$, then we find a hyperplane $I'$ of $I$ through $\<K_1,K_2,R_4\>$ not containing $K_3$ and we are done. So we may assume that $K_3$ is contained in $\<K_1,K_2,R_4\>$. As in the previous case (for instance, considering a $k$-space in $R_4$), it is easy to see that we can find a hyperplane $I_1$ of $I$ containing $K_1,K_2$, not containing $K_3$ and intersecting $R_4$ in a $k$-space $K_4$. Notice that $K_4$ is not contained in $\<K_1,K_2\>$ by assumption, so we can find an $(i-2)$-space $I'_1$ containing $\<K_1,K_2\>$ and intersecting $K_4$ in a $(k-1)$-space $K'_4$.    Hence $I'_1\cap J_1=K_1$, $I'_1\cap J_2=K_2$, $\dim (I'_1\cap J_3)\leq k-1$ and $\dim(I'_1\cap J_4)=k-1$. In $\Res(I_1')$, the spaces $\<J_1,I_1'\>/I_1'$, $\<J_2,I_1'\>/I_1'$ and $\<J_4,I_1'\>/I_1'$ have dimension $j-k-1$, $j-k-1$ and $j-k$, respectively. Since $j-k<n-i$, we can find a point  in $\Res(I'_1)$, which avoids these three subspaces, and also avoids $\<J_3,I_1'\>/I_1'$ in case $\dim(I_1'\cap J_3)=k-1$ (the union of two hyperplanes and two subhyperplanes is never the entire space for $|\L|>2$). That point corresponds to an $(i-1)$-space $I'$, which does the trick, interchanging the roles of $J_3$ and $J_4$.

%Now $\dim\<K_1,K_2,R_4\>\leq 3k+2<k+i+1<2i+1\leq 2j+1\leq n$, hence there is a point $p$ not contained in $\<K_1,K_2,R_4\>$. The $(i-1)$-space $I'=\<I'_1,p\>$ does the trick, switching the roles of $J_3$ and $J_4$.

So in all three cases, we found an $(i-1)$-space $I'$, and in $\Res(I')$, the spaces $\<J_1,I'\>/I'$, $\<J_2,I'\>/I'$ and $\<J_3,I'\>/I'$ have dimension $j-k-1$, $j-k-1$ and $j-k$, respectively. As in the third case above, we can find a point in $\Res(I')$ which avoids these three subspaces, and which also avoids $\<J_4,I'\>/I'$ in case $\dim(I'\cap J_4)=k-1$. This point corresponds to an $i$-space $I^*$ intersecting $J_1$ and $J_2$ in $k$-spaces, intersecting $J_3$ in a $(k-1)$-space, and intersecting $J_4$ in a space of dimension distinct from $k$. This contradiction to $\{J_1,J_2,J_3,J_4\}$ being a $\Gamma$-round-up quadruple concludes the proof of the lemma.    
\end{proof}

% Before we finish the proof of Main Result~\ref{main1}, we show that, if $\{J_1,J_2,J_3,J_4\}$ is a $\Gamma$-round-up quadruple, then we can choose indices such that $\dim (J_1\cap J_2)\geq k$. 
 
% \begin{lemma}\label{lem3}
%If $\{J_1,J_2,J_3,J_4\}$ is a $\Gamma$-round-up quadruple, then we can choose indices such that $\dim (J_1\cap J_2)\geq k$. 
 %\end{lemma}

%\begin{proof}
%Suppose $\dim(J_a\cap J_b)<k$, for all $a,b\in\{1,2,3,4\}$, $a\neq b$. Since $2k\leq i-1\leq j-1$, we can find a $k$-space $K_1$ in $J_1$ disjoint from $J_2\cup J_3\cup J_4$. Likewise, we can find three pairwise disjoint $k$-spaces $K_2,K_2',K_2''$ in $J_2$ disjoint from $J_1\cup J_3\cup J_4$. By Lemma~\ref{lem2}. we can choose indices and dashes such that $\<K_1,K_2\>$ and $\<K_1,K_2'\>$ intersect $J_3$ in $k$-spaces $K_3$ and $K_3'$, respectively. Now $K_3$ and $K_3'$ are contained in $\<K_1,K_2,K_2'\>$, which is $3k+2$-dimensional. Clearly, $K_3$ and $K_3'$ are disjoint, as otherwise their intersection contains a point of $K_1$, contradicting the choice of $K_1$ not containing a point of $J_3$. Hence a dimension argument inside $\<K_1,K_2,K_2'\>$ implies that $\<K_2,K_2'\>$ and $\<K_3,K_3'\>$ intersect in a space of dimension at least $2(2k+1)-(3k+2)=k$, contradicting the fact that $\dim (J_2\cap J_3)<k$. 
%\end{proof}

We now show that the $J_a$, $a\in\{1,2,3,4\}$, pairwise intersect in a common space. This is the analogue of Lemma~\ref{lem1}.

\begin{lemma}\label{lem4}
If $\{J_1,J_2,J_3,J_4\}$ is a $\Gamma$-round-up quadruple, then $J_1\cap J_2=J_2\cap J_3=J_3\cap J_4=J_1\cap J_3=J_2\cap J_4=J_1\cap J_4$.  
\end{lemma}

\begin{proof}
Let the maximum dimension of $J_a\cap J_b$, $\{a,b\}\subseteq\{1,2,3,4\}$, $a\neq b$, be $\ell$.  There are two possibilities.

\begin{enumerate}
\item Suppose $\ell\geq k$. We may assume $\dim(J_1\cap J_2)=\ell$. By Lemma~\ref{lem2}, every $k$-space contained in $J_1\cap J_2$ is also contained in $J_3\cup J_4$. Since $J_1\cap J_2$ cannot be the union of $J_1\cap J_2\cap J_3$ and $J_1\cap J_2\cap J_4$ unless one of the latter two spaces coincides with $J_1\cap J_2$, we may assume that $J_1\cap J_2\subseteq J_3$. By maximality of $\ell$, we already have $J_1\cap J_2= J_3\cap J_1=J_2\cap J_3$. Now we choose a $k$-space $K$ in $J_1\cap J_2$ and a point $p\in J_3\setminus J_1$. Applying Lemma~\ref{choose} in $\Res(\<K,p\>)$, we find an $i$-space $I_p$ containing $\<K,p\>$, intersecting $J_1$ and $J_2$ precisely in $K$, intersecting $J_3$ in $\<K,p\>$, and disjoint from $J_4\setminus J_3$ (note that $\dim\Res(\<K,p\>)=n-k-2$, $\dim\<J_1,p\>/\<K,p\>=\dim\<J_2,\>/\<K,p\>=j-k-1$, $\dim J_3/\<K,p\>=j-k-2$, $\dim \<K,J_4,p\>/\<K,p\>\leq j$, and $\dim I_p/\<K,p\>=i-k-2$). 

%We claim that we can find an $i$-space $I_p$ containing $\<K,p\>$, intersecting $J_1$ and $J_2$ precisely in $K$, intersecting $J_3$ in $\<K,p\>$, and disjoint from $J_4\setminus J_3$. Indeed, put $K'=\<K,p\>$, then in $\Res(K')$, which is $(n-k-2)$-dimensional, we need to find an $(i-k-2)$-space disjoint from the $(j-k-1)$-spaces $S_1:=\<J_1,p\>/K'$ and $S_2:=\<J_2,\>/K'$, from  the $(j-k-2)$-space $S_3:=J_3/K$, and from the space $S_4:=$, which has dimension at most $j$. Since the four spaces $S_1,S_2,S_3,S_4$ have codimension at least $2,2,3,1$, respectively, we can always find a point $x$ outside their union. Repeating this argument in $\Res(\<K',x\>)$, and then subsequently still another $i-k-3$ times, we find $I_p$, proving our claim.  

It follows that $I_p\cap J_4$ is a $k$-space. But $I_p\cap J_4\subseteq \<K,p\>$, by construction. So $J_4\cap J_3$ contains a $k$-space contained in $\<K,p\>$.  If $J_2\cap J_3\cap J_4$ has dimension $\leq\ell-2$, then we can choose $K$ such that $K\cap J_4$ is at most $(k-2)$-dimensional, a contradiction. Hence $J_2\cap J_3\cap J_4$ has dimension $\ell-1$ (if it had dimension $\ell$, the lemma was proved). So $\<J_2\cap J_3,J_3\cap J_4\>$ has dimension at most $\ell+1$. If $\dim\<J_2\cap J_3,J_3\cap J_4\><j$, then by picking $p$ outside    $\<J_2\cap J_3,J_3\cap J_4\>$, we again obtain a contradiction. Hence $\ell+1=j$ and so $\dim (J_3\cap J_4)=\ell$. By maximality of $\ell$ we can repeat the argument of the previous paragraph, and conclude that either $J_1$ or $J_2$ must contain $J_3\cap J_4$, clearly a contradiction.  Hence $J_2\cap J_3\cap J_4$ is $\ell$-dimensional after all, proving the lemma in this case.
\item Suppose $\ell<k$. Let $m\leq\ell$ be the maximal dimension of $J_1\cap J_a$, $a=2,3,4$ and suppose $\dim(J_1\cap J_2)=m$. Let $K_b$ be any $k$-space in $J_b$ containing $J_1\cap J_2$, $b=1,2$. By Lemma~\ref{lem2}, the subspace $\<K_1,K_2\>$ intersects one of $J_3,J_4$, say $J_3$, in a $k$-subspace $K_3$.  As in the second part of the proof of Lemma~\ref{lem1}, we deduce $J_1\cap J_3=K_1\cap K_3$. If this equals $J_1\cap J_2$, then fine. Otherwise we choose another $k$-space $K_1'$ in $J_1$ containing $J_1\cap J_2$ and not containing $K_1\cap K_3$. Then either the space $\<K_1',K_2\>$ intersects $J_3$ in a $k$-space (and then the same argument as in the second part of the proof of Lemma~\ref{lem1} leads to $J_1\cap J_2=J_2\cap J_3=J_3\cap J_1$), or $\<K_1',K_2\>$ intersects $J_4$ in a $k$-space $K_4$. Then $J_1\cap J_4=K_1'\cap K_4$. If the latter equals $J_1\cap J_2$, then fine again. Otherwise we choose a $k$-space $K_1''$ in $J_1$ containing $J_1\cap J_2$ and neither containing $K_1\cap K_3$ nor $K_1'\cap K_4$. Then $\<K_1'',K_2\>$ intersects one of $J_3,J_4$, say $J_3$, in a $k$-space, and the same argument as above leads to $J_1\cap J_2=J_2\cap J_3=J_3\cap J_1$. 

Hence in any case we may assume that $J_1\cap J_2=J_2\cap J_3=J_3\cap J_1$ (possibly by interchanging the roles of $J_3$ and $J_4$). Now let $f$ be the maximal dimension of $J_4\cap J_c$, $c=1,2,3$. If $f=-1$ and $m\geq 0$, then the second part of the proof of Lemma~\ref{lem5} can be copied and applied to $J_1,J_2,J_3$ and thus yields a contradiction. If $f=m=-1=\ell$, then the lemma is proved. Hence we may assume that $f\geq 0$ and $m\geq 0$ by symmetry. By the arguments above (interchanging $J_1$ and $J_4$, and $f$ and $m$), we may assume that $J_2\cap J_3=J_3\cap J_4=J_4\cap J_2$. But then $m=f=\ell$ and by maximality of $\ell$, $J_1\cap J_4$ is not bigger than $J_1\cap J_2$, hence coincides with it.
\end{enumerate}
The lemma is proved. 
\end{proof}

We can modify the proof of Lemma~\ref{lem5} to prove the following lemma.

\begin{lemma}\label{lem6}
Every $\Gamma$-round-up quadruple $\{J_1,J_2,J_3,J_4\}$ is a regular round-up quadruple.
\end{lemma}

\begin{proof}
By Lemma~\ref{lem4}, we know $J_a\cap J_b=J_c\cap J_d$, for all $a,b,c,d\in\{1,2,3,4\}$, $a\neq b$ and $c\neq d$.  Put $\ell=\dim (J_1\cap J_2)$. There again are two possibilities to consider separately. 

\begin{enumerate}
\item Assume $\ell\geq k$. Let $p_a\in J_a\setminus J_3$ be arbitrary, $a=1,2$. For $a=1,2$, choose a $k$-space $K_a$ containing $p_a$ and intersecting $J_1\cap J_2$ in a $(k-1)$-space, with the latter independent of $a$. Lemma~\ref{lem2} implies that the $(k+1)$-space $\<K_1,K_2\>$ intersects either $J_3$ or $J_4$ in $k$-space $K$ intersecting $J_1\cap J_2$ in $K_1\cap K_2$. Hence the line $p_1p_2$ intersects $K$, and hence either $J_3$ or $J_4$, in a point. Lemma~\ref{lem0} shows the assertion.  
\item As in the proof of Lemma~\ref{lem5}, we prove that the situation $\ell<k$ leads to a contradiction. 
Indeed,  the argument is similar to the argument in the second part of the proof of Lemma~\ref{lem5}, but we consider three $k$-subspaces $K_1,K_1',K_1''$ in $J_1$ through $J_1\cap J_2$ contained in a common $(k+1)$-space, choose a $k$-space $K_2\subseteq J_2$ with $J_1\cap J_2\subseteq K_2$ and note that at least two of $\<K_1,K_2\>$, $\<K_1',K_2\>$, $\<K_1'',K_2\>$ intersect the same subspace among $J_3,J_4$ in $k$-spaces. Then we argue as in the second part of the proof of Lemma~\ref{lem5} and derive a contradiction.
\end{enumerate}
This completes the proof of the lemma.
\end{proof}

As in the previous subsection, we can now again define $\Gamma_1$ with vertex set the set of $j$-spaces, and adjacency being contained in a common $\Gamma$-round-up quadruple. The graph $\Gamma_1$  has two distinguishable kinds of maximal cliques  which we use to define the graph $\Gamma_2$,  isomorphic to $\Gamma^n_{[j-1,j+1]}(\L)$. Recall that $\Gamma_2$ is tripartite and that the triparts correspond to $j$-subspaces and the two kinds of maximal cliques in $\Gamma_1$. We define a \emph{$j$-Grassmann line} as the set of $j$-spaces contained in a given $\Gamma$-round-up quadruple $Q$, or forming such a quadruple with any three members of $Q$. Alternatively, these are the intersections of a member of one class with a member of the other class of maximal cliques of $\Gamma_1$ if this intersection contains at least two elements. 

Let $I$ be an $i$-space intersecting some $j$-space $J$ of a given maximal clique $Q$ in a $k$-space $K$. There are four possibilities.

\begin{itemize}
\item[(1)] The clique $Q$ consists of all $j$-spaces through a given $(j-1)$-space $J'$ and $K\subseteq J'$. Then $I$ is $\Gamma$-adjacent to either none, all, or all but one members of an arbitrary $j$-Grassmann line in $Q$.  
\item[(2)]  The clique $Q$ consists of all $j$-spaces through a given $(j-1)$-space $J'$ and $K\cap J'$ has dimension $k-1$. Then $I$ is $\Gamma$-adjacent to either none, exactly one, or all members of an arbitrary $j$-Grassmann line in $Q$.  
\item[(3)] The clique $Q$ consists of all $j$-spaces in a given $(j+1)$-space $J''$ and $\dim(I\cap J'')=k+1$. Then $I$ is $\Gamma$-adjacent to either none, all, or all but one members of an arbitrary $j$-Grassmann line in $Q$.  
\item[(4)] The clique $Q$ consists of all $j$-spaces in a given $(j+1)$-space $J''$ and $I\cap J''=K$. Then $I$ is $\Gamma$-adjacent to either none, exactly one, or all members of an arbitrary $j$-Grassmann line in $Q$.  
\end{itemize}

If we define the bipartite graph with vertices at the one hand the $i$-spaces and at the other hand the maximal cliques of fixed type, and we define an $i$-space $I$ and a clique $Q$ to be adjacent if  $I$ is $\Gamma$-adjacent to either none, all, or all but one members of an arbitrary $j$-Grassmann line in $Q$, then in Case (1) we denote the graph $\Gamma^{-1}_{\forall_{-1}}$ and it is isomorphic to $\Gamma^n_{i,j-1;k}(\L)$; in Case (3) we denote it by $\Gamma^{+1}_{\forall_{-1}}$ and it is isomorphic to $\Gamma^n_{i,j+1;k+1}(\L)$. 

If we now define an $i$-space $I$ and a clique $Q$ to be adjacent if  $I$ is $\Gamma$-adjacent to either none, exactly one, or all members of an arbitrary $j$-Grassmann line in $Q$, then in Case (2) we denote the graph $\Gamma^{-1}_{\exists_{1}}$ and it is isomorphic to $\Gamma^n_{i,j-1;k-1}(\L)$; in Case (4) we denote it by $\Gamma^{+1}_{\exists_{1}}$ and it is isomorphic to $\Gamma^n_{i,j+1;k}(\L)$. 

We can again repeat these four constructions to the appropriate graphs and obtain, with self-explaining notation, the following isomorphisms.
 $$\begin{array}{ll}\Gamma^{-m}_{\exists_1}\cong \Gamma^n_{i,j-m;k-m}(\L), \hspace{1cm} &\Gamma^{+m}_{\exists_1}\cong\Gamma^n_{i,j+m,k}(\L),\\ [3mm] \Gamma^{-m}_{\forall_{-1}}\cong\Gamma^n_{i,j-m,k}(\L),& \Gamma^{+m}_{\forall_{-1}}\cong\Gamma^n_{i,j+m,k+m}(\L),\end{array}$$
for all natural $m$ for which these make sense. 

The smallest $m$ for which the next graph in the series $\Gamma^{-m'}_{\exists_1}$ cannot be defined (since every $i$-space adjacent to at least one member of the appropriate maximal clique, is always adjacent to all but one members of a $(j-m-1)$-Grassmann line), is $m=k+1$. Similarly, the smallest $m$ for which the next graph in the series $\Gamma^{+m'}_{\exists_1}$ becomes void is $m=n+k-i-j$, since an $i$-space and an $(n+k-i+1)$-space inside an $n$-space always meet in at least a $(k+1)$-space (and this does not hold for the previous step!). As in the previous section, this determines $k$ and $n-i-j$ uniquely. 

Now completely similar to the previous section, the smallest value of $m$ for which $\Gamma^{-m}_{\forall_{-1}}$ becomes an empty graph is $j-k+1$, and the smallest value of $m$ for which $\Gamma^{+m}_{\forall_{-1}}$ becomes empty    is $i-k+1$. This determines $i$ and $j$, and so we determined all values $i,j,k,n$. 

The proofs of the Main Results are complete now. 

We mention the following special case of Main Result~\ref{main1}. Suppose $-1\leq k\leq j\leq n-1$, $j\neq -1$ and recall that $\Gamma^n_{j;k}(\L)$ denotes the graph of $j$-spaces of $\PG(n,\L)$ where two $j$-spaces are adjacent if they intersect in a $k$-space.

\begin{cor}\label{th2}
Let $\L$ and $\L'$ be two skew fields, and let $-1\leq k\leq j\leq n-1$, $-1\leq k'\leq j'\leq n'-1$ be integers, with $-1\notin\{j,j'\}$ and $n\geq 2$. 
\begin{itemize}
%\item[$(i)$] If $i=k$, then $\Gamma^n_{i,j;\geq k}(\L)\cong \Gamma^n_{i,j;k}(\L)$.
%\item[$(ii)$] If $k=0$, then $\Gamma^n_{i,j;\geq k}(\L)$ is the bipartite complement of $\Gamma^n_{i,j;-1}(\L)$.
\item[$(i)$] If $j=n-1$ and $k=n-2$, or $j=0$ and $k=-1$, then $\Gamma^n_{j;k}(\L)$ is a complete graph.
\item[$(ii)$]  If $n+k< 2j$ or $j=k$, then $\Gamma^n_{j;k}(\L)$ is an empty graph.
\item[$(iii)$] If $\L\cong\L'$, $n=n'$, $j'=n-1-j$ and $k'=n-1+k-2j$, then $\Gamma^n_{j;k}(\L)\cong \Gamma^{n'}_{j';k'}(\L')$.
%\item[$(iv)$] If $k=2j+1-n$, then $\Gamma^n_{j;\geq k}(\L)$ is the complement of $\Gamma^n_{j;k-1}(\L)$.
%\item[$(v)$] If $\L=\L'$, $n=n'$, $i'=n-1-j$, $j'=n-1-i$ and $k'=n-1+k-i-j$, then $\Gamma^n_{i,j;\geq k}(\L)\cong \Gamma^{n'}_{i',j';\geq k'}(\L')$.
\item[$(iv)$] If $2j\leq n-1$, $-1\leq k<j$, $(k,j)\neq(-1,0)$ and $2j'\leq n'-1$, then $\Gamma^n_{j;k}(\L)\cong\Gamma^{n'}_{j'; k'}(\L')$ if and only if $\L\cong \L'$ and $(j,k,n)=(j',k',n')$. In this case every graph isomorphism is induced by a semi-linear bijection from $V_n(\L)$ to $V_{n'}(\L')$, or possibly to $V_{n'}^*(\L')$ (the dual of $V_{n'}(\L')$) if $2j=n-1$ and $\L'\cong (\L')^*$ (the latter is the opposite skew field). 
\end{itemize} 
\end{cor}

\begin{proof}
Again by taking bipartite double of the graphs in question and then applying Main Result~\ref{main1}.
\end{proof}

\begin{rem}\em\label{surjections}
In the statements of all our results, we may substitute ``graph isomorphism'' by ``graph epimorphism'', where \emph{epimorphism} means a surjective mapping on the vertices such that two vertices in the source graph are adjacent if and only if their images are adjacent. Indeed, if there would exist a non-bijective graph epimorphism, then there are two vertices in the source graph with exactly the same neighborhood set. But it is trivial to see that this is impossible for all the graphs we defined, as soon as they are nontrivial.  
\end{rem}

\section{The finite and the thin case}
\subsection{The finite case}
In the finite case, there is a shortcut to prove our main results, though very un-algorithmic, but rather efficient.

It follows from the main result and Table VI in \cite{Lie-Pre-Sax:87} that $\mathsf{P\Gamma L}(n+1,q)$, extended with a duality whenever possible, is a maximal subgroup of the symmetric or alternating group $\mathsf{Sym}(N),\mathsf{Alt}(N)$, where $N$ is the number of $j$-spaces of $\PG(n,q)$, except possibly when $q\in\{2,3\}$ and $N=q^{m-1}(q^{m}-1)/(2,q-1)$. But in this case, $N$ can never be the number of $j$-spaces of $\PG(n,q)$, for any $j$, since this number is never divisible by $q$. Hence, if the automorphism group of the graph $\Gamma^n_{i,j;k}(q)$ or $\Gamma^n_{i,j;\geq k}(q)$, with the appropriate restrictions to make the graph nontrivial, were larger than the corresponding group induced by semi-linear permutations, then the full symmetric or alternating group would act on at least one of the bipartition classes. We now claim that this implies that the graph is complete bipartite, which is a contradiction. Indeed, we have the following lemma.

\begin{lemma}
Let $v_1,v_2$ be two distinct vertices of one of the nontrivial graphs $\Gamma^n_{i,j;k}(q)$ or $\Gamma^n_{i,j;\geq k}(q)$ in the same bipart, then there are at least two neighbors of $v_1$ not adjacent to $v_2$.
\end{lemma}

\begin{proof}
For $\Gamma^n_{i,j;k}(q)$, $k>-1$, and $\Gamma^n_{i,j;\geq k}(q)$, with $i+j\leq n+1$, this follows immediately from the easy fact that, for every pair $v_1,v_2 $of $j$-spaces ($i$-spaces) there exists a pair $w_1,w_2$ of $k$-spaces contained in $v_1$ but not in $v_2$. For $\Gamma^n_{i,j;-1}(q)$, this follows from the previous fact for $\Gamma^n_{i,j;0}(q)$ and with the roles of $v_1$ and $v_2$ interchanged.  
\end{proof}

Now let $v$ be a vertex of  $\Gamma^n_{i,j;k}(q)$ or $\Gamma^n_{i,j;\geq k}(q)$. Let $w$ be a neighbor of $v$, and let $w'$ be an arbitrary vertex in the same bipart as $w$, but not adjacent to $v$. Let $\theta$ be an automorphism of the graph stabilizing the set of neighbors of $v$ distinct from $w$, and mapping $w$ onto $w'$ (which exists since the alternating group acts as an automorphism group on each bipart). By the previous lemma, $v$ is fixed, and so $v$ is adjacent to $w'$. Varying $v$ and $w'$, we see that the graph is complete bipartite, proving our claim. 

This in particular proves Main Result~\ref{main1} for $|\L|=q=2$. The graphs $\Gamma^n_{i,j;k}(2)$ and $\Gamma^n_{i,j;\geq k}(2)$, with the appropriate restrictions, are not isomorphic to any other graph of type III because it has a different automorphism group. These graphs are mutually non-isomorphic because they pairwise have different orders (cardinalities of the biparts)  and bi-valences (we leave the details to the reader).

\subsection{The finite thin case}\label{finitethin}
The same argument works in the \emph{thin} case, except that $\mathsf{Sym}(2n-1)$ acting on all $(n-1)$-subsets of a $(2n-1)$-set is not maximal in $\mathsf{Sym}(\frac{1}{2}\binom{2n}{n})$ (indeed, it is contained in $\mathsf{Sym}(2n)$ acting on partitions of sizes $n,n$ of a $2n$-set), giving rise to a counter example, and there are three other small exceptions, one of which also gives rise to a counter example. We now provide the details.

Let $\Omega^{(n)}$ be an $n$-set and let $\Omega_j$ be the set of all $j$-subsets of $\Omega$, $0\leq j\leq n$. Suppose $i,j,k$ are integers with $0\leq k\leq i\leq j\leq n-1$ and let $\Gamma^n_{i,j;\geq k}$ and $\Gamma^n_{i,j;k}$ be the bipartite graphs with biparts $\Omega_j$ and $\Omega_i$, where a $j$-subset is adjacent to an $i$-subset if their intersection is of size at least $k$ and precisely $k$, respectively. For convenience, we also denote $\Gamma^n_{i,j;k}$ by $\Gamma^n_{j,i;k}$ and $\Gamma^n_{i,j;\geq k}$ by $\Gamma^n_{j,i;\geq k}$, $0\leq k\leq i\leq j\leq n-1$. 

\begin{theorem}\label{thethin}
Let $0\leq k\leq i\leq j\leq n-1$ and $0\leq k'\leq i'\leq j'\leq n'-1$ be integers with $0\notin\{i,i'\}$ and $n\geq 2$. 
\begin{itemize}
\item[$(i)$] If $i=j=k\geq 1$, then $\Gamma^n_{i,j;k}$ and $\Gamma^n_{i,j;\geq k}$ are matchings; if $i+j=n$ and $k=0$, then $\Gamma^n_{i,j;k}$ is a matching.
\item[$(ii)$] If $i=j=n-1=k+1$ or $i=j=1=k+1$, then $\Gamma^n_{i,j;k}$ is the complement of a matching.
\item[$(iii)$]  If $n+k<i+j$, then $\Gamma^n_{i,j;k}$ is an empty graph;  if $n+k\leq i+j$ or $k=0$, then $\Gamma^n_{i,j;\geq k}$ is a complete bipartite graph.
\item[$(iv)$] If $k=i+j+1-n$, then $\Gamma^n_{i,j;\geq k}$ is the bipartite  complement of $\Gamma^n_{i,j;k-1}$; if $i=k$, then $\Gamma^n_{i,j;\geq k}\cong\Gamma^n_{i,j;k}$.
\item[$(v)$] If $n=n'$,  $j'=n-j$ and $k'=i-k$, then $\Gamma^n_{i,j;k}\cong \Gamma^{n'}_{i,j';k'}$ and $\Gamma^n_{i,j;\geq k}$ is the bipartite complement of $\Gamma^{n'}_{i,j';\geq k'+1}$.
\item[$(vi)$] If $i,j\leq n/2$, $k<j$, $k\leq i/2$ if $j=n/2$,  $i',j'\leq n'/2$, $k\leq i'/2$ if $j'=n'/2$, then $\Gamma^n_{i,j;k}\cong\Gamma^{n'}_{i',j';k'}$ if and only if $(i,j,k,n)=(i',j',k',n')$. In this case every graph isomorphism is induced by a bijection from $\Omega^{(n)}$ to $\Omega^{(n')}$ (this is also trivially true for $(i,j,k)=(i',j',k')=(1,1,1)$), except if $(n,i,j,k)=(6,2,2,1)$, where one can compose with an arbitrary collineation belonging to $\mathsf{L}_4(2)$ acting on the points of $\PG(3,2)$, and except if $(n,i,j,k)=(4n^*-1,2n^*-1,2n^*-1,n^*-1)$, with $n^*\geq 2$ an integer, where the full automorphism group is $\mathsf{Sym}(4n^*)$ acting on the $(2n^*,2n^*)$-partitions of a $4n^*$-set. 
\item[$(vii)$] If $i,j\leq n/2$, $0\neq k<j$, $k\leq (i+1)/2$ if $j=n/2$, $i',j'\leq n'/2$, and $k\leq(i'+1)/2$ if $j'=n'/2$, then $\Gamma^n_{i,j;\geq k}\cong\Gamma^{n'}_{i',j';\geq k'}$ if and only if $(i,j,k,n)=(i',j',k',n')$. In this case every graph isomorphism is induced by a bijection from $\Omega^{(n)}$ to $\Omega^{(n')}$ (the latter is also trivially true for $(i,j,k)=(i',j',k')=(1,1,1)$). 
\end{itemize}
\end{theorem}

\begin{proof}
The assertions $(i)$ up to $(v)$ are easy to check. We now prove $(vi)$ and $(vii)$.
As before, it suffices, under the given conditions, to reconstruct $(i,j,k,n)$ from the graphs $\Gamma^n_{i,j;k}$ and $\Gamma^n_{i,j;\geq k}$. 

Suppose we know that the automorphism group of $\Gamma^n_{i,j;k}$ or $\Gamma^n_{i,j;\geq k}$ is $\mathsf{Sym}(n)$. Then we already know $n$. From the sizes of the biparts we deduce $i$ and $j$, and then the bi-valence of the graph reveals $k$. Hence it suffices to show that the automorphism group of $\Gamma^n_{i,j;k}$ and $\Gamma^n_{i,j;\geq k}$ is $\mathsf{Sym}(n)$.

Noting that no two vertices of $\Gamma^n_{i,j;k}$ or $\Gamma^n_{i,j;\geq k}$ have exactly the same set of neighbors, it suffices to show that the automorphism group induced on one of the biparts is $\mathsf{Sym}(n)$. This follows immediately if $\mathsf{Sym}(n)$ is a maximal subgroup of the symmetric group acting on $\binom{n}{j}$ or $\binom{n}{i}$ letters. Now according to \cite{Lie-Pre-Sax:87}, this is the case, except possibly in the following cases (under the restrictions of $(vi)$ and $(vii)$ of Theorem~\ref{thethin}):
$$(n,i,j)\in\{(6,2,2),(10,3,3),(12,4,4),(2\ell-1,\ell-1,\ell-1):\ell\in\mathbb{N}, \ell\geq 3\}.$$
\begin{itemize}
\item $(n,i,j)=(6,2,2)$.

This is an easy case, as up to taking bipartite complement, we may suppose we have $\Gamma^6_{2,2;0}$ or $\Gamma^6_{2,2;1}$ (and we can distinguish these by their valences, which are 6 and 8, respectively). For $\Gamma^6_{2,2;0}$,  two vertices of the same bipart correspond to disjoint pairs if the have exactly one neighbor, and to intersecting pairs if the have three common neighbors. If we define a graph on one bipart by declaring two vertices adjacent if they have three common neighbors in $\Gamma^6_{2,2;0}$, then we can recover the $3$-subsets as the maximal cliques of size 3 of that graph. Hence we can derive $\Gamma^6_{2,3;2}$, which is not in our list of exceptions, and the result follows. For $\Gamma^6_{2,2;1}$, we note that the vertices of one bipart adjacent to a vertex of the other bipart can be considered as the points of the unique generalized quadrangle of order $2$ opposite (non-collinear to) a given point. In $\PG(3,2)$, these sets are just all complements of planes, hence $\mathsf{L}_4(2)$ acts on $\Gamma^6_{2,2;1}$, and $\mathsf{Sym}(6)$ is contained in it as $\mathsf{Sp}_4(2)$. 

\item $(n,i,j)=(10,3,3)$.

In the case there are essentially four different cases: $\Gamma^{10}_{3,3;k}$, for $k=0,1,2$, and $\Gamma^{10}_{3,3;\geq 2}$. These graphs can again be distinguished by their valences (these are 35, 63, 21 and 22, respectively). In the cases $k=0,2$, the number of common neighbors of two triples in the same bipart is equal to $4,10$ or $20$, and to $0,4$ or $8$, respectively, according to whether the triples have exactly no, exactly one, or exactly two points in common. Hence on each bipart the structure of $\Gamma^{10}_{3;2}$ can be recovered. The maximal cliques of size 8 in that graph correspond with pairs of $\Omega(10)$, and so we can uniquely build $\Gamma^{10}_{3,2;2}$, which is not in our list of possible exceptions. 

If $k=1$, an interesting phenomenon occurs: the graph on any bipart where two vertices are adjacent when they have exactly 30 neighbors is a strongly regular graph $(120,63,30,36)$. Hence we can only recover $\Gamma^{10}_{3;1}$. But the maximal cliques of size 7 of that graph clearly correspond to $7$-sets of $\Omega(10)$ (where they induce a Fano plane), and there is one vertex not adjacent to any seven of these vertices. Such a vertex corresponds to a triple which is disjoint from all triples in the maximal clique. So we can recover when two triples are disjoint, after all, and hence when they share 2 elements, too. The argument of the previous paragraph now applies. 

Concerning $\Gamma^{10}_{3,3;\geq 2}$, the number of common neighbors of two triples in the same bipart is equal to $0,4$ or $10$, respectively, according to whether the triples have exactly no, exactly one, or exactly two points in common. The same argument as above now applies. 

\item $(n,i,j)=(12,4,4)$.

Here are essentially six possibilities, namely $\Gamma^{12}_{4,4;k}$, $k=0,1,2,3$, and $\Gamma^{12}_{4,4;\geq \ell}$, $\ell=2,3$.  The valences again tell us these cases apart (the valences are 70, 224, 168, 32, 201 and 33, respectively). In the following table, where the rows and columns of the upper $4\times 4$ square are numbered from $0$ to $3$, the $(i,j)$-entry is the number of common neighbors of two vertices of the same bipart in $\Gamma^{12}_{4,4;i}$ which correspond to 4-sets of $\Omega(12)$ intersecting in a $j$-set. The two bottom rows tell us in the $j$th position exactly the same for the graphs $\Gamma^{12}_{4,4;\geq \ell}$, $\ell=2,3$. The labels are written on the left and the top and are self-explaining. 
\begin{center}
\begin{tabular}{c|cccc}
& 0 & 1 & 2 & 3 \\
\hline
0 & 1 & 5 & 15 & 35 \\
1 & 96 & 100 & 100 & 126\\
2 & 36 & 54 & 64 & 84\\
3 & 0 & 0 & 4 & 10\\
$\geq 2$ & 36 & 72 & 102 & 138\\
$\geq 3$ & 0 & 0 &  4 & 12
\end{tabular}
\end{center}

It follows from the last column of this table that in each bipart we can uniquely and unambiguously construct $\Gamma^{12}_{4;3}$ (the vertices are the $4$-sets of $\Omega^{(12)}$ and two $4$-sets are adjacent if they intersect in a $3$-set). It is easy to see that we can identify the $3$-subsets with the maximal cliques of $\Gamma^{12}_{4;3}$ consisting of 9 vertices. Hence we can build $\Gamma^{12}_{3,4;3}$, which is not in the list of possible exceptions. 

\item $(n,i,j)=(2\ell-1,\ell-1,\ell-1)$, $\ell\in\mathbb{N}, \ell\geq 3$.

In this case, it follows from Table III of  \cite{Lie-Pre-Sax:87} that the only way in which $\mathsf{Sym}(2\ell-1)$ is not the full automorphism group of $\Gamma^n_{i,j;k}$ or $\Gamma^n_{i,j;\geq k}$ is when $\mathsf{Sym}(2\ell)$ is the full automorphism group and we can represent the vertices of our graph as the $(\ell,\ell)$-partitions of $\Omega^{(2\ell)}$, with ``induced'' adjacency. We now take a closer look at this induced adjacency. Suppose two $(\ell-1)$-sets of $\Omega^{(2\ell-1)}$ intersect in a set of size $m$. Then the partition classes of the corresponding $(\ell,\ell)$-partitions of $\Omega^{(2\ell)}$ intersect in sets of sizes $m+1$ and $\ell-m-1$. Conversely, if   the partition classes of two $(\ell,\ell)$-partitions of $\Omega^{(2\ell)}$ intersect in sets of sizes $m+1$ and $\ell-m-1$, then the corresponding $(\ell-1)$-sets of $\Omega^{(2\ell-1)}$ intersect in a set of size either $m$ or $\ell-m-2$. 

From this we easily deduce that the automorphism group of a graph $\Gamma^n_{i,j;k}$ or $\Gamma^n_{i,j;\geq k}$ is $\mathsf{Sym}(2\ell)$ if and only if whenever two vertices are adjacent as soon as the corresponding subsets intersect in a set of size $m$, then the two vertices corresponding to subsets intersecting in a set of size $\ell-m-2$ are also adjacent. This is never the case for $\Gamma^n_{i,j;\geq k}$ and only the case for $\Gamma^n_{i,j;k}$ if $k=\ell-k-2$; hence if $k=(\ell-2)/2$.  
\end{itemize}
This completes the proof of the theorem.
\end{proof}

Similarly to Corollaries~\ref{th1} and~\ref{th2}, one has also a corollary to Theorem~\ref{thethin}. We will not explicitly mention it, but the reader can easily state it for himself. We content ourselves by mentioning that all automorphisms of the graphs $\Gamma^n_{j;\geq k}$ where vertices are the $j$-subsets of $\Omega^{(n)}$, adjacent if they share at least $k$ elements ($1\leq k<j\leq n-1$) and $\Gamma^n_{j;k}$ with same vertex set but adjacency now defined as the intersection exactly being a $k$-subset, are induced by ordinary permutations of $\Omega^{(n)}$, except that the automorphism group of $\Gamma^{(4n^*-1)}_{2n^*-1;n^*-1}$ is $\mathsf{Sym}(4n^*)$.  The graph $\Gamma^6_{2;1}$ is \emph{not} an exception here because we can consider the extended bipartite double, which is the bipartite complement of $\Gamma^6_{2,2;0}$, and the latter has no exceptional behavior. Note, however, that the ordinary bipartite double of $\Gamma^6_{2;1}$ is $\Gamma^6_{2,2;1}$, and this pair constitutes an example of a connected non-bipartite graph whose automorphism group is much smaller than that of its bipartite double (usually the size of the automorphism group of the bipartite double is just twice the size of the original graph). 

\subsection{The infinite thin case}
Finally, we take a look at the analogue of the thin case for an infinite set. 
Let $0\leq k\leq i\leq j$, with $k<j$ be integers and let $\Omega$ be any infinite set. Let the graph $\Gamma^{\Omega}_{i,j;k}$ be the bipartite graph with vertices the $i$-subsets and the $j$-subsets of $\Omega$, where an $i$-set is adjacent to a $j$-set if they intersect in a $k$-set. Similarly we define $\Gamma^{\Omega}_{i,j;\geq k}$ in the obvious way, and also $\Gamma^\Omega_{j;k}$ and $\Gamma^\Omega_{j;\geq k}$. 

We first show the following.

\begin{prop}\label{prop1}
The elements of $\Omega$ are recoverable from $\Gamma^{\Omega}_{j;k}$, $0< k<j$. In particular, every graph automorphism is induced by a permutation of $\Omega$.
\end{prop}

\begin{proof}
We first reconstruct the $k$-subsets of $\Omega$. We  claim that an infinite clique in the graph corresponds to an infinite number of $j$-subsets containing a fixed $k$-subset. 

Indeed, let $C$ be an infinite clique and suppose by way of contradiction that there are three members of the clique, say $J_1,J_2,J_3$ with $J_1\cap J_2\neq J_2\cap J_3$. Let $J$ be an arbitrary member of the clique distinct from $J_1,J_2,J_3$. If $|J\cap(J_1\cup J_2\cup J_3)|=k$, then clearly this $k$-subset must be contained in $J_1\cap J_2\cap J_3$, a contradiction. Hence $|J\cap(J_1\cup J_2\cup J_3)|>k$. We now have infinitely many choices for $J$ and only a finite number of choices (bounded by $2^{3j}$) for the intersection $J\cap(J_1\cup J_2\cup J_3)$. It follows that there are two $j$-subsets $J,J'$ in the clique intersecting $J_1\cap J_2\cap J_3$ in the same set. But then $|J\cap J'|>k$, a contradiction. 

Now the transitive closure of the relation ``shares at least two members'' among the infinite cliques of the graph defines an equivalence relation whose classes are in natural bijective correspondence with the $k$-subsets. Hence we can reconstruct the bipartite graph $\Gamma^\Omega_{k,j}$ of   $k$-sets and $j$-sets, where adjacency is given by symmetrized inclusion. Note that the valence of a vertex corresponding to a $j$-subset is $\binom{j}{k}$.

We can recognize the $(k+1)$-subsets as the sets with minimal $>1$  number of common neighbors of two $j$-subsets (and then they have exactly $k+1$ common neighbors). This way we determine $k$, and since we know $\binom{j}{k}$ already, we also know $j$. But we also recognize $\Gamma^\Omega_{k+1,k}$. Now, two vertices representing $(k+1)$-subsets are at distance $2k$ from one another if and only if they intersect in a single element. This relation defines $\Gamma^\Omega_{k+1;1}$. As in the beginning of this proof, we can now reconstruct $\Gamma^\Omega_{k+1,1}$, telling us exactly which elements are contained in each $(k+1)$-set.

This proves the proposition. 
\end{proof}

Now let $0\leq k\leq i\leq j$, with $k<j$ and let $\Gamma$ be either isomorphic to $\Gamma^{\Omega}_{i,j;k}$ or isomorphic to $\Gamma^{\Omega}_{i,j;\geq k}$ (where $\Omega$ is an infinite set of any cardinality). Let $
\cA$ be an infinite set of vertices belonging to the same bipart of $\Gamma$. Then we say that $
\cA$ is an \emph{$\ell$-star}, $\ell\in\mathbb{N}\setminus\{0\}$, if it satisfies the following properties.
\begin{itemize}
\item[(S1)] If $v$ is a vertex belonging to the other bipart, then $v$ is either adjacent to exactly $\ell+1$ members of $\cA$, or not adjacent to exactly $\ell$ members of $\cA$, or adjacent to no members of $\cA$.
\item[(S2)] For every subset $A\subset\cA$ containing exactly $\ell+1$ elements, there exists a finite nonzero number of vertices $v$ of $\Gamma$ such that $v$ is adjacent to every member of $A$ and not adjacent to every member of $\cA\setminus A$.  
\item[(S3)] For every subset $B\subset\cA$ containing exactly $\ell$ elements, there exists a finite nonzero number of vertices $v$ of $\Gamma$ such that $v$ is not adjacent to every member of $B$ and adjacent to every member of $\cA\setminus B$.  
\end{itemize}
Also, we say that $\cA$ is an \emph{$\ell$-flower}, $\ell\in\mathbb{N}\setminus\{0\}$, if it satisfies the following properties.
\begin{itemize}
\item[(F1)] If $v$ is a vertex belonging to the other bipart, then $v$ is either adjacent to exactly $\ell+1$ members of $\cA$, or adjacent to all members of $\cA$, or adjacent to no members of $\cA$.
\item[(F2)] For every subset $A\subset\cA$ containing exactly $\ell+1$ elements, there exists a finite nonzero number of vertices $v$ of $\Gamma$ such that $v$ is adjacent to every member of $A$ and not adjacent to every member of $\cA\setminus A$.  
\item[(F3)] There are infinitely many vertices adjacent to all members of $\cA$.
%\item[(S3)] For every subset $B\subset\cA$ containing exactly $\ell$ elements, there exists a constant finite number of vertices $v$ of $\Gamma$ such that $v$ is not adjacent to every member of $B$ and adjacent to every member of $\cA\setminus B$.  
\end{itemize}
Finally, $\cA$ will be called an \emph{$i$-regular star} if it consists of the vertices corresponding to all $i$-sets of $\Omega$  containing a fixed $(i-1)$-set. Similarly for a $j$-regular star. If $i=1$, then an $i$-regular star coincides with $\Omega$.

Basically, we want to show that the graphs $\Gamma^{\Omega}_{i,j;k}$ contain unique $\ell$-stars, with $\ell\in\{j-k,i-k\}$, which are either $i$-regular or $j$-regular stars. Moreover these graphs do not contain any $\ell$-flowers. Also, we want to show that  the graphs $\Gamma^{\Omega}_{i,j;\geq k}$ contain unique $\ell$-flowers, with $\ell\in\{j-k,i-k\}$, which are either $i$-regular or $j$-regular stars. Moreover these graphs do not contain any $\ell$-stars. But first we isolate some more-or-less trivial cases. 

Given $\Gamma$ (with $k\leq i\leq j$, $k<j$), we can easily decide whether $k=i$ or not; indeed, $k=i$ if and only if the valence of the vertices of one bipart is finite. In this case, we make a new graph by throwing away the bipart with infinite valence and declare two remaining vertices adjacent when they have a unique common neighbor in $\Gamma$. The graph thus obtained is clearly isomorphic to $\Gamma^\Omega_{j;i}$, and we can apply Proposition~\ref{prop1}. Hence every graph automorphism of  $\Gamma^{\Omega}_{i,j;\geq i}\cong \Gamma^{\Omega}_{i,j;i}$ is induced by a permutation of $\Omega$.

Also, one checks that the only cases in which there are vertices with finite covalence (meaning that there are only a finite number of vertices of the other bipart not adjacent to a vertex of one bipart) are the cases of the graphs $\Gamma^\Omega_{1,j;0}$, $j\geq 1$. But in these cases, clearly every graph automorphism of the bipartite complement is induced by a permutation of $\Omega$. 

So from now on, we may assume that no vertex has finite valence nor finite covalence. 

Finally, we can isolate the graphs $\Gamma\cong\Gamma^\Omega_{i,j;0}$ (case $k=0$) from the rest as follows.

\begin{lemma}\label{complement}
In each of the graphs $\Gamma^\Omega_{i,j;k}$ and $\Gamma^\Omega_{i,j;\geq k}$ with $k>0$, there exists in each bipart a finite set of vertices not adjacent with a common vertex. But in the graph $\Gamma^\Omega_{i,j;0}$ each finite set of vertices contained in one of the biparts is adjacent to some vertex.   
\end{lemma}

\begin{proof}
In the bipart of the $i$-subsets one can consider a set $S$ of $j+1$ disjoint $i$-subsets. Then an arbitrary $j$-subset can intersect at most $j$ members of $S$ nontrivially. Hence, if $k>0$, no vertex is adjacent to all members of $S$. The same thing is true for $i$ and $j$ interchanged. However, if we consider any finite number of $i$-subsets, then we can always find a $j$-subset disjoint from all these $i$-subsets. And the same thing holds for $i$ and $j$ interchanged. This proves the lemma. 
\end{proof}

If $k=0$, then we consider the complement of the graph $\Gamma^\Omega_{i,j;0}$, which is $\Gamma^\Omega_{i,j;\geq 1}$. So we may assume $k>0$ from now on.

\begin{prop}\label{propS}
If $\Gamma\cong\Gamma^\Omega_{i,j;k}$, with $1\leq k<i\leq j$, then every $\ell$-star, $\ell\geq 0$, is either a $j$-regular star, or an $i$-regular star (and both occur), and $\ell=i-k$ or $\ell=j-k$, respectively. Also, $\Gamma$ does not contain $\ell$-flowers, for any natural $\ell$.  
\end{prop}

\begin{proof}
To fix the ideas, we assume that $\cA$ is an $\ell$-star or $\ell$-flower contained in the bipart corresponding to the $j$-sets. We will cease to assume $i\leq j$ in order to treat the case of $i$-sets at the same time (but we do assume $k<j$). 

By assumptions (S3) and (F3), there is an $i$-set $I$ adjacent to every member of  an infinite subset $\cA_0$ of $\cA$. Since $k>0$, since every member of $\cA_0$ intersects $I$ in exactly $k$ elements of $\Omega$, and since $I$ contains a finite number of $k$-subsets, there must be some $k$-subset $K\subseteq I$ contained in every member of an infinite subset $\cA_K\subseteq\cA_0$. Let $X\supseteq K$ be a set with the property that the set $\cA_X$ of elements of $\cA_K$ that contain $X$ has infinitely many elements. Clearly $|X|\leq j-1$, and $K$ has this property. Hence we can pick a maximal set $X$ with that property. It follows that for any $p\in\Omega\setminus X$, the number of elements of $\cA_X$ containing $p$ is a finite number $n_p$. An easy consequence of that fact is that for every finite subset $C\subseteq\Omega\setminus X$, there exist infinitely many members of $\cA_X$ missing $C$, and only a finite number of members of $\cA_X$ meet $C$ nontrivially. Hence a straightforward inductive argument implies that  for any positive integer $m$, we can find a subset $A_m\subseteq\cA_X$ of cardinality $m$ such that no member of $\cA_X$ intersects two members of $A_m$ nontrivially outside $X$, and such that $A_m$ contains an arbitrary given member of $\cA_X$.

We consider $A_{i-k+1}$ and consider an $i$-set $I$ such that $|I\cap X|=k-1$ and $|I\cap J|=k$, for every $J\in A$. Then clearly, $I$ is not adjacent with infinitely many members of $\cA_X$, and adjacent with at least $i-k+1\neq0$ members of  $\cA_X$. It follows from (S1) and (F1) that $I$ is adjacent with exactly $\ell+1$ members of $\cA$, hence with at most $\ell+1$ members of $\cA_X$. This implies $i-k\leq \ell$. 

Now we consider $A_{i-k}$ and consider an $i$-set $I'$ such that $|I'\cap X|=k$ and $|I'\cap J|=k+1$, for every $J\in A$. Then clearly, $I'$ is adjacent with infinitely many members of $\cA_X$, and not adjacent with at least $i-k\neq 0$ members of $\cA_X$. This contradicts (F1), hence $\Gamma$ cannot contain an $\ell$-flower. 

Consider $A_{\ell}$. Then, by (S3), there is an $i$-set $I''$ not adjacent to all members of $A_{\ell}$ and adjacent to everything else. It is easy to see that the latter implies $|I''\cap X|=k$. The former implies that we can distribute the $i-k$ elements of $I''\setminus X$ among the members of $A_\ell$ such that every member gets at least one element. Hence $i-k\geq\ell$, which yields $i-k=\ell$. Also, it now follows that $I''$ contains exactly one element in $J\setminus X$, for each $J\in A_\ell$. Varying this element over $J\setminus X$ (leading to different $i$-sets not adjacent with exactly the same set of members of $\cA$), we see that for each point $p\in J\setminus X$, we have $n_p=1$. Since we could choose one of the elements of $A_\ell$ completely arbitrarily, we conclude that $n_p\in\{0,1\}$, for all $p\in\Omega\setminus X$. 

If $n_p=1$,  for all $p\in\Omega\setminus X$, then $\cA_X$ is an $(i-k)$-star and a $j$-regular star, and this implies easily $\cA=\cA_X$, which proves the assertion.  

So we may assume that there exists a point $p\in\Omega\setminus X$ with $n_p=0$. Then clearly $\cA\neq\cA_X$ as we can easily produce an $i$-set intersecting $X$ in a $(k-1)$-set, containing a point $p$ with $n_p=0$; and containing a point $q$ with $n_q=1$ (this is possible since $i\geq k+1$); this $i$-set is adjacent to at least one and at most $i-k$ vertices of $\cA_X$. 

So let $J\in\cA\setminus\cA_X$. Then there is a point $x\in X\setminus J$. We choose a $k$-set $K$ in $X$ containing $x$. Then we choose $i-k$ members of $\cA_X$ disjoint from $J$. In each of these members, we choose a point not in $X$. The union of these points with $K$ is an $i$-set $I^*$ which   is not adjacent with exactly $i-k$ members of $\cA_X$; but $I^*$ is not adjacent with $J$, too, a contradiction. 

The assertion is proved.  
\end{proof}

\begin{prop}\label{propF}
If $\Gamma\cong\Gamma^\Omega_{i,j;\geq k}$, with $1\leq k<i\leq j$, then every $\ell$-flower, $\ell\geq 1$, is either a $j$-regular star, or an $i$-regular star (and both occur), and $\ell=i-k$ or $\ell=j-k$, respectively. Also, $\Gamma$ does not contain $\ell$-stars, for any natural $\ell$.  
\end{prop}

\begin{proof}
Let $\cA$ be an $\ell$-flower or $\ell$-star in $\Gamma$.  As before we assume that it consists of $j$-sets, and we drop the assumption $i\leq j$. Similarly as in the proof of Proposition~\ref{propS}, there exists a set $X\subseteq\Omega$, with $k\leq|X|\leq j-1$, such that the set $\cA_X$ of members of $\cA$ containing $X$ is infinite, and every element $p\in\Omega\setminus X$ is contained in finitely many members of $\cA$.   

Select $\ell$ members of $\cA_X$ arbitrarily. If $\cA$ is an $\ell$-star, then by (S3), there exists an $i$-set $I$ adjacent to every member of $\cA_X$ except for the selected $\ell$ members. It is easily seen that the fact that $I$ is adjacent with infinitely members of $\cA_X$ implies that $|I\cap X|\geq k$. But then $I$ is adjacent with every member of $\cA_X$, a contradiction to $\ell\geq 1$. Hence $\cA$ is an $\ell$-flower.

Suppose now that there exists $J\in\cA\setminus\cA_X$. Then we can close a $k$-subset in $X$ intersecting $J$ in less than $k$ elements. We add $i-k$ elements of $\Omega\setminus(X\cup J)$ and obtain an $i$-set adjacent to every element of $\cA_X$, but not to $J$, contradicting (F1). Hence $\cA=\cA_X$. 

Now we again define $n_p$ as the number of members of $\cA$ containing $p$, $p\in\Omega\setminus X$. Similarly as in the proof of Proposition~\ref{propS} (case $k\geq 1$), one shows $i-k\leq\ell$. Also, we can select a set $A$ of $\ell+1$ members of $\cA$ pairwise intersecting in only $X$. Then (F2) implies that there is some $i$-set $I'$ adjacent to all members of $A$, and to no member of $\cA\setminus A$. The latter implies that $|I'\cap X|\leq k-1$. Since every member of $A$ must intersect $I'$ in at least $k$ elements, the maximum value for $\ell+1$ is $i-(k-1)$; hence $\ell=i-k$. Also, since we can choose one member of $A$ completely arbitrarily in $\cA$, we deduce that $n_p=1$ as soon as $n_p\neq 0$, for every $p\in\Omega\setminus X$. By deleting from $I'$ an element outside $X$ and replacing it with some element $q$, also outside $X$, but with $n_q=0$, we obtain a contradiction and have hence shown that $n_p=1$, for all $p\in\Omega\setminus X$.

We now claim $|X|=j-1$. Indeed, assume $|X|\leq j-2$. Let $J\in\cA$. We choose an $i$-set $I''$ with exactly $k-1$ elements in $X$ and $k+1$ elements in $J$. Then  $I''$ is adjacent to at least one, but at most $i-(k-2)$ members of $\cA$, a contradiction. 

This completes the proof of the proposition.
\end{proof}

We can now show the following theorems.

\begin{theorem}
Let $0\leq k\leq i\leq j$, with $k<j$ and let $\Gamma$ be either isomorphic to $\Gamma^{\Omega}_{i,j;k}$ or isomorphic to $\Gamma^{\Omega}_{i,j;\geq k}$, $k>0$ (where $\Omega$ is an infinite set of any cardinality). Then every graph automorphism of $\Gamma$ is induced by a permutation of $\Omega$. Also, all graphs $\Gamma^{\Omega}_{i,j;k}$ and $\Gamma^{\Omega}_{i,j;\geq k}$ (with the given restrictions on the parameters) are pairwise non-isomorphic.
\end{theorem}

\begin{proof}
We already discussed the cases in which there are vertices with either finite valence or finite covalence. Hence we may assume that no vertex has finite valence or finite covalence. Also, if for every finite set of one of the biparts, there exists a vertex adjacent to all members of that set, then we know by Lemma~\ref{complement} that $k=0$, and we go on with the bipartite complement. Then Propositions~\ref{propS} and~\ref{propF} imply that the only $\ell$-stars and $\ell$-flowers of $\Gamma$ are $j$-regular stars and $i$-regular stars, which are $(i-k)$-stars or -flowers and $(j-k)$-stars or -flowers, respectively. This determines $i-k$ and $j-k$ (if we went on with the bipartite complement then this already determines all parameters of the original graph $\Gamma^\Omega_{i,j;0}$). We consider one bipart $B$, say containing the vertices of the $(i-k)$-flowers or -stars, and define a new bipartite graph $\Gamma_i$ with $B$ as one bipart, and the $(i-k)$-stars or -flowers as other bipart. Adjacency is containment made symmetric. Then $\Gamma'\cong\Gamma^\Omega_{j-1,j;j-1}$ and the finite valence reveals $j$.  Also, the assertion about the automorphism group now follows from  the discussion preceding Proposition~\ref{propS}.
\end{proof}

\begin{theorem}
Let $0\leq k<i$ and let $\Gamma$ be either isomorphic to $\Gamma^{\Omega}_{i;k}$ or isomorphic to $\Gamma^{\Omega}_{i;\geq k}$, $k>0$ if $i>1$ (where $\Omega$ is an infinite set of any cardinality). Then every graph automorphism of $\Gamma$ is induced by a permutation of $\Omega$. Also, all graphs $\Gamma^{\Omega}_{i;k}$ and $\Gamma^{\Omega}_{i;\geq k}$ (with the given restrictions on the parameters) are pairwise non-isomorphic.
\end{theorem}

\begin{proof}
If $i=1$, this is trivial (and we also have trivial graphs). In the other cases, the graphs are not trivial, and the theorem follows from the previous one by taking the bipartite double. 
\end{proof}

Again, in the above results, we may assume graph-epimorphisms instead of automorphisms, see also Remark~\ref{surjections}.

\end{document}